\documentclass[11pt,a4paper]{article}
\usepackage{amsgen, amsmath, amsfonts, amsthm, amssymb,enumerate,amscd,graphics,dsfont,tikz-cd,tikz}

\usetikzlibrary{intersections}

\usepackage{changes} 

\newtheorem{defn}{Definition}[section]
\newtheorem{prop}[defn]{Proposition}
\newtheorem{lem}[defn]{Lemma}
\newtheorem{thm}[defn]{Theorem}
\newtheorem{cor}[defn]{Corollary}
\newtheorem{rem}[defn]{Remark}
\newtheorem {conj}[defn]{Conjecture}

\newcommand {\ZZ}{{\mathds Z}}

\newcommand {\WW}{{\mathcal W}}

\newcommand {\XX}{{\mathcal X}}

\newcommand {\E}{{\mathcal E}}
\newcommand {\C}{{\mathds C}}

\newcommand {\A}{{\mathcal A}}

\newcommand {\Q}{{\mathds Q}}
\newcommand {\QQ}{{\bf Q}}
\newcommand {\R}{{\mathds R}}
\newcommand {\OO}{{\mathcal O}}

\newcommand {\HH}{{\mathds H}}
\newcommand {\UH}{{\mathcal H}}

\newcommand {\LA}{{\lambda}}

\newcommand {\M}{{\mathcal M}}

\newcommand {\CP}{{\mathds P}}

\newcommand {\D}{{\mathcal D}}

\newcommand {\MS}{{{\mathfrak S}}}

\newcommand {\RE}{{{\mathrm E}}}

\newcommand {\GG}{{\mathcal G}}

\newcommand {\TCC}{\tilde{\mathcal C}}

\newcommand {\TK}{{\tilde{K}}}

\newcommand {\io}{{\iota}}

\newcommand {\lb}{{\lambda}}
\newcommand {\KE}{{K_{E_1 \times E_2}}}
\newcommand {\TKE}{{{\tilde K}_{E_1\times E_2}}}
\newcommand {\CPP}{\CP^1 \times \CP^1}

\newcommand{\TC}{\tilde{C}}

\def\sgn{\operatorname{sgn}}

\def\Pic{\operatorname{Pic}}

\def\div{\operatorname{div}}
\def\deg{\operatorname{deg}}

\def\dim{\operatorname{dim}}

\def\Im{\operatorname{Im}}
\def\ord{\operatorname{ord}}
\def\reg{\operatorname{reg}}

\newcommand {\es}{{\epsilon^*_{\sgn}}}

\tikzset{
  symbol/.style={
    draw=none,
    every to/.append style={
      edge node={node [sloped, allow upside down, auto=false]{$#1$}}}
  }
}

\DeclareMathOperator{\divsr}{div} 
\DeclareMathOperator{\CH}{CH}
\DeclareMathOperator{\SL}{SL}
\DeclareMathOperator{\Mp}{Mp}

\def\skipchunk#1{}

\title{Algebraic cycles and values of Green's functions : Products of Elliptic Curves}
\author{Ramesh Sreekantan (with an appendix by Kannappan Sampath)}

\begin{document}
	\baselineskip=17pt

\maketitle
\begin{abstract}
	
	Gross and Zagier defined certain `higher Green's functions' on products of modular curves and conjectured that the value of these functions at complex multiplication points should be logarithms of algebraic numbers. This is now a theorem of Li \cite{li} and Bruinier-Li-Yang \cite{BLY}.
	
	We relate this conjecture to the existence of motivic cycles in the universal family of products of elliptic curves along the lines of Mellit \cite{mell} and Zhang \cite{zhan}. Using this we are able to prove Zagier's conjecture in some cases when the two CM points have the same discriminant. This is originally a theorem of Viazovska \cite{viaz}. 
		
	 Li, Bruinier-Li-Yang, Bruinier-Ehlen-Yang \cite{BEY}, Viazovska \cite{viaz} and others relate this conjecture to Borcherds' lifts of weakly holomorphic modular forms. Their works, coupled with ours,  suggest that there should be a link between motivic cycles in the universal family on the one hand and Borcherds lifts on the other. We explain why this is the case. This suggests a motivic interpretation of weakly holomorphic modular forms.  In the special case we look at we show that indeed this is true. 
	
\vspace{\baselineskip}

\end{abstract}

\tableofcontents

\section{Introduction}

\subsection{The conjecture of Gross and Zagier}

A conjecture of Gross and Zagier \cite{grza}, \cite{GKZ}, \cite{zagiICM} asserts that the values of certain `higher Green's functions' at CM points are logarithms of algebraic numbers. 

We recall the conjecture in its schematic form (leaving precise definitions to Section~\ref{greensfunctions}):

\begin{conj} Let $\Gamma$ be a congruence subgroup of $\SL_2(\mathbb{Z})$ and let $X=\overline {\Gamma\backslash \HH}$ denote the corresponding compactified modular curve. Let 
	$$f(z)=\sum_{m \in \ZZ} c_f(m)e^{2\pi iz}$$
	be a weakly holomorphic modular form of weight $(-2j)$ for $\Gamma$ with $c_f(-m) \in \ZZ$ for all $m>0$. Let 
	$$G_{j+1}^f(z_1,z_2)=\sum c(-m)m^jG_{j+1}(z_1,T_m z_2)$$
	denote the higher Green's function associated to $f$, where $T_m$ is the $m^{th}$ Hecke correspondence. Let $T_f=\bigcup_{c(-m)\neq 0} T_m$ be the singular locus of $G^f_{j+1}$.

	Let $z_1$ and $z_2$ be CM points of discriminants $D_1$ and $D_2$  respectively lying on $X\times X \backslash T_f$. Then there exists an $\alpha  \in \bar{\Q}$  depending only on $D_1, D_2, j$ and $f$  such that 
	$$(D_1 D_2)^{j/2} G^f_{j+1}(z_1,z_2)=\log|\alpha|.$$

\end{conj}

This conjecture is now a theorem of Li \cite{li} and Bruinier-Li-Yang \cite{BLY}, though conditional results were found by Gross, Kohnen and Zagier \cite{GKZ}, Zhang \cite{zhan}, Mellit \cite{mell}, Viazovska\cite{viaz}, Zhou \cite{zhou} and Bruinier-Ehlen-Yang \cite{BEY}.

The approach of Zhang and Mellit is to link the values of higher Green's functions to regulators of motivic cycles in the generic fibre of the universal family while the approach of Li, Bruinier-Ehlen-Yang and Viazovska is to link the higher Green's functions themselves to Borcherds lifts of weakly holomorphic modular forms. 

That there should be a link between motivic cycles and Borcherds lifts was speculated upon in \cite{sree2000} following the remarkable work of Borcherds \cite{borc} proving the theorem of Gross-Zagier, Hirzebruch-Zagier and van der Geer in one fell swoop. There has been recent progress on this approach in the work of Doran-Kerr \cite{doke} and Brown-Fonseca \cite{brfo}. 

\subsection{Motivic cycles}

In this paper, we construct explicit motivic cycles in the group $H^3_{\M}(A_{\eta},\Q(2))$ where $A_{\eta}$ is the generic fibre of the universal family over the moduli of products of elliptic curves. More precisely, we construct the cycles in the associated Kummer $K3$ surface $\TK_{A_{\eta}}$. 

We have the following theorem (Lemma \ref{lem-higher-chow-cycles-on-KmK3}, Theorem \ref{mainthm})
\begin{thm} 
	
	Let $\A \to X(2) \times X(2)$ be the universal abelian surface that is split as product of two elliptic curves with full level $2$ structure. Over a point $(\lb_1,\lb_2)$ the fibre $A_{\lb_1,\lb_2}$ is isomorphic to $E_{\lb_1}\times E_{\lb_2}$ with $E_{\lambda_i}$ elliptic curve with Legendre parameter $\lambda_i$.
		
	 Let $T_n^{\gamma}$ be a component of the Hecke correspondence corresponding to those pairs $(E,E')$ such that there is an isogeny  of odd degree $n$ from $E \rightarrow E'$.  If we fix such an isogeny $\phi_{\eta}:E_{T^{\gamma}_{n,\eta}} \rightarrow E'_{T{\gamma}_{n,\eta}}$ of the generic fibre over $T^{\gamma}_n$ then there exists a motivic cycle $\xi_n^{\phi,\gamma}$  in the group $H^3_{\M}(A_U,\Q(2))$, where $U=(X \times X )\backslash T_n^{\gamma}$ satisfying the following properties.
	 
	 \begin{itemize}
	 	\item $\xi_n^{\phi,\gamma}$ is indecomposable in the sense of Section \ref{indecomp}.
	
	\item  If $\partial$ denotes the boundary map in the localization sequence (Section \ref{nonarchregulator}) then, 
	$$\partial(\xi_{n}^{\phi,\gamma})=\Gamma_{\phi_{T^{\gamma}_n}}-\Gamma_{-\phi_{T_n^{\gamma}}}$$
	where $\Gamma_{\phi_{T_n^{\gamma}}}$ denotes the closure in the fibre over $T_{n}^{\gamma}$ of the graph of the isogeny $\phi_{\eta}$
	
	\end{itemize}

	\end{thm} 

\skipchunk{\begin{thm} Let $E_1=E_{\lambda_1}$ and $E_2=E_{\lambda_2}$ be two elliptic curves with full level two structure and let $\KE$ and $\TKE$ be the Kummer surface and Kummer $K3$ surface associated to $E_1 \times E_2$ as in Section \ref{sssec-km-surface-prod-ell-curves}. 
Let  $\phi:E_1 \rightarrow E_2$ is an isogeny of degree $n$ with $n$ and let $T_n^{\gamma}$ be a component of  the $n^{th}$ Hecke correspondence parametrizing those pairs $(E_1,E_2)$ where such a $\phi$ exists. Then 
\begin{itemize}
	\item There is a family of  irreducible nodal curves $C_{n,1}=C_{n,1,\lambda_1,\lambda_2}^{\phi,\gamma}$ on $\KE$ with nodes at $2n+1$ specified points on $\KE$ defined as long as $\lb_1,\lb_2$ {\bf do not} lie on $T_n^{\gamma}$.  
	\item   For every node $c=c_{\lb_1,\lb_2}$ of $\KE$ lying on $C_{n,1}$ there exists an   motivic cycle 
$$\xi^{\phi,\gamma}_{n,c}=\xi_{n,c} \in H^3_{\M}(\TKE,\Q(2))$$   
and this determines a family of cycles parametrized by $\lb_1,\lb_2$ and defined as long as $\lb_1,\lb_2$ does not lie on $T_n^{\gamma}$ 
\item This family of cycles is generically  indecomposable in the sense of Section \ref{indecomp}. 
\end{itemize}
\end{thm}}

The construction uses a curious link with the enumerative geometry of rational curves in $\CPP$.

\subsection{Algebraicity Results}

In the second part of this paper  we prove a theorem (Theorem \ref{algebraicity}) which  shows that the regulators of indecomposable cycles imply algebraicity for values of the higher Green's functions of degree $2$ at CM points:

\begin{thm}
	\label{alg}
	Let $X \times X$ be the self product of a modular curve and $A_{\eta}$ the generic fibre of the universal Abelian surface. Let $\xi=\sum(C_i,f_i)$ be an indecomposable cycle in the group $H^{2j+1}_{\M}(A^j_{\eta},\Q(j+1))$ with boundary $\partial(\xi)=\sum_T a_T Z^j_T$, where $Z^j_T$ is the Hecke cycle (Section \ref{heckecycles}) over the component $T$ of a Hecke correspondence $T_{n_T}$

	If $(z_1,z_2)$ is a pair of CM points with the {\em same discriminant} such that $(z_1,z_2)\in X \times X\backslash \bigcup_{a_T \neq 0} T$, then  
	$$\frac{4^jD^j}{\binom{2j}{j}}\sum a_T n^j_TG_{j+1}(z_1,T z_2)=\log  \prod \left| f_i(C_i|_{A^j_{z_1,z_2}} \cap Z^j_{z_1,z_2})\right|$$
	where by $f$ evaluated at a divisor $\sum m_P P$ we mean $\prod f(P)^{m_P}$ and $Z^j_{z_1,z_2}$ is the CM cycle (Section \ref{cmcycles}) in the fibre $A^j_{z_1,z_2}$ over $(z_1,z_2)$. In particular, since $C_i|_{z_1,z_2},Z_{z_1,z_2}$ and $f_i$ are defined over $\bar{\Q}$ the value is the logarithm of an algebraic number. 
\end{thm}
We then apply this to our construction of cycles to obtain algebraicity results. We have the following theorem (Theorem \ref{zagierconjecture})

\begin{thm}  Let $T_n$ be the Hecke correspondence with $n$ odd on $X(2) \times X(2)$. Let $z_1,z_2$ be two CM points with CM by the same field of discriminant $D$ with $(z_1,z_2)\notin T_n$.	Then 
	$$2DnG_2(z_1,T_nz_2)=\log|\alpha|$$
	where $\alpha$ is an algebraic number. In particular, since there are no cusp forms of weight $4$ for $\Gamma(2)$,  the conjecture of Gross and Zagier holds in this special case. 
\end{thm}

If we restrict our cycles and Green's functions to a Hecke correspondence $T_m \subset X(2) \times X(2)$ then one can obtain some results for mixed discriminant CM points along the lines of the work of Mellit \cite{mell}.

The two approaches to this theorem suggests that there should be a link between motivic cycles, on the one hand, and weakly holomorphic modular forms, on the other. This suggests the following conjecture: 

\begin{conj}
	Let $f=\sum c(m)q^m$ be a weakly holomorphic modular form of weight $-2j$ for $\Gamma \subset SL_2(\ZZ)$. Let $A^j_{\eta}$ be the generic fibre of the self product of the universal Abelian surface $E_{1,\eta} \times E_{2,\eta}$ over $X(\Gamma) \times X(\Gamma)$.

	Then, there is a correspondence 
	$$f \leftrightarrow \xi_f$$ 
	where 
	$$\xi_f \in  H^{2j+1}_{\M}(A^{j}_{\eta},\Q(j+1))$$
	satisfies 
	$$\partial(\xi_f)=\sum_{m>0} c(-m) Z^j_{T_m}$$ 
	where $Z^j_{T_m}$ is a Hecke cycle (Section\ref{heckecycles}). Further, the {\bf motivic} Green's function determined by the regulator of $\xi_f$ coincides with the {\bf automorphic} Green's function of Bruinier-Li-Yang: 	
	$$\langle \reg(\xi_f),\eta^j_{z_1,z_2}\rangle=\Phi^j((z_1,z_2),f)$$
	where $\reg$ is as in Section \ref{regulators}, $\eta^j_{z_1,z_2}$ is  as in Section \ref{heckecycles} and $\Phi^j((z_1,z_2),f)$ is a suitable  Borcherds lift.  
\end{conj}

In a companion paper \cite{sree-simple} we construct motivic cycles  in the case when the Abelian surface is simple. This should lead to analogous results of algebraicity for Bruinier's automorphic Green's  functions \cite{brui}. In a related paper \cite{sree-products} we construct cycles in the group $H^5_{\M}(A^2_{\eta},\Q(3))$ which imply similar results for $G_3$. 

\subsection *{Acknowledgements.} 

I would like to thank Jan Bruinier, Ken Sato, Shohei Ma, Vijay Patankar, Yingkun Li, Jens Funke, Kirti Joshi, Subham Sarkar among many others for their comments and suggestions. I would especially like to thank Kannappan Sampath for his help with the appendix and for his careful reading. I would like to thank ISI Bangalore and KAIST, Daejeon for their hospitality when this work was done. I would also like to acknowledge the use of AI - specifically Google Gemini Pro and ChatGPT - in clarifying a lot of arguments and determining constants. Finally I would like to thank the referee of an earlier version of this paper for their very useful comments.

\part{Motivic cycles on Products of Elliptic curves}

\section{Products of Elliptic Curves}

\subsection{Elliptic curves with full level $2$ structure.}

All the material in this section can be found in any standard textbook on elliptic curves, for instance Silverman \cite{silv}. 

Let $\Gamma(1)=\SL_2(\ZZ)$ and let $\Gamma(2)$ be the principal congruence subgroup of $\SL_2(\ZZ)$ of level $2$. Let $Y(1)=\Gamma(1)\backslash \HH=\CP^1-\{\infty\}$. This space parametrizes elliptic curves over $\C$. Let $X(1) \simeq \CP^1$ denote its canonical compactification.

Let $E$ be an elliptic curve with full level $2$ structure. Such a curve can be put in {\bf Legendre form} 
$$E_{\LA}:Y^2=X(X-1)(X-\LA)$$
with the two torsion points being $(0,0),(1,0),(\infty,0)$ and $(\LA,0)$.

The moduli of such elliptic curves is the open modular curve $Y(2)=\Gamma(2)\backslash \HH$ which is $\CP^1-\{0,1,\infty\}$. The natural forgetful map $\nu: Y(2) \to Y(1)$ realizes $Y(2)$ as a six-sheeted cover of $Y(1)$. Let $X(2)\simeq \CP^1$ denote its canonical compactification.

The map $\nu:X(2)\longrightarrow X(1)$ is given by 
$$\nu(\lb)= j(E_{\lb})=256 \frac{(1-\lb+\lb^2)^3}{\lb^2(1-\lb)^2}$$
and is ramified at $0,1$ and $\infty$ which all map to the cusp $\infty$ on $X(1)$. If 
 $E_{\lb_1}$ and $E_{\lb_2}$ are isomorphic, then 
 $$\lb_2 \in \left\{\lb_1,1-\lb_1, \frac{\lb_1-1}{\lb_1},\frac{1}{\lb_1},\frac{1}{1-\lb_1},\frac{\lb_1}{\lb_1-1} \right\}.$$

 \subsection{The Kummer surface of a product of elliptic curves}
 
 \subsubsection{The Kummer curve of an elliptic curve}

The {\bf Kummer curve} $K_{E_{\LA}}$ of $E_{\LA}$ is the curve $E_{\LA}/\iota_{\LA}$, where $\iota_{\LA}$ is the involution \begin{align*}
\iota_{\LA}:E_{\LA} &\longrightarrow E_{\LA}\\
x &\mapsto -x
\end{align*}
 This is $\CP^1$ with $4$ distinguished points $0,1,\infty$ and $\lambda$  and can be thought of as the tangent line at $\LA$ to the curve $X(2)$. Let $\pi_{\lambda}:E_{\lambda} \rightarrow K_{E_{\lambda}}$ denote the map determined by $\iota$. Explicitly, over a point $x$ the fibre is 
 $$\pi^{-1}_{\lb}(x)=\{(x,y),(x,-y)\}$$
 where $y^2=x(x-1)(x-\lb)$. This is ramified at the four 2-torsion points.

\subsubsection{The Kummer surface of a product of elliptic curves}\label{sssec-km-surface-prod-ell-curves}

If $A$ is an Abelian surface  and $\io_A$ is the involution $\io_A(x)=-x$ then the {\bf Kummer surface} $K_A$ is the surface $A/\io_A$. This surface has sixteen nodes corresponding to the sixteen two-torsion points which are the fixed points of $\io_A$. 
 
In the case when $A=E_1 \times E_2$ is the product of two elliptic curves, where $E_i=E_{\lb_i}$, this has the following description. Let $K_{E_1}$ and $K_{E_2}$ be the two Kummer curves. 

The map 
$$E_1 \times E_2 \longrightarrow K_{E_1} \times K_{E_2}$$
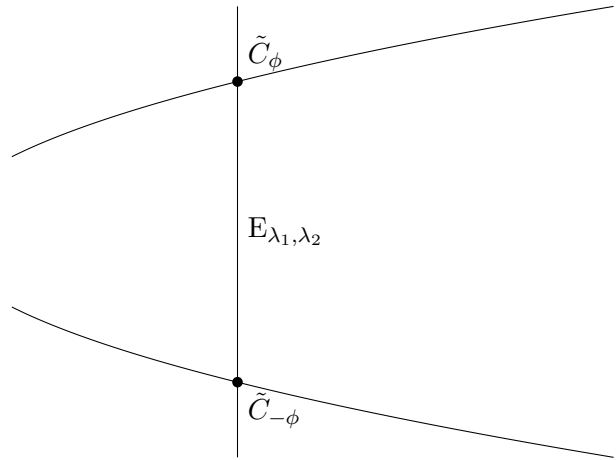
is a four sheeted cover of $\CP^1 \times \CP^1$. Let $P_x,P_{-x}$ be the two points over $x$ under the map $E_i \rightarrow K_{E_i}$. The fibre over the a point $(x,x') \in  K_{E_1} \times K_{E_2}$ consists of the four points 
$$\{(P_x,P_{x'}), (P_x,P_{-x'}), (P_{-x},P_{x'}), (P_{-x},P_{-x'})\}$$
The quotient by the involution $\io_{E_1 \times E_2}=\io_{E_1} \times \io_{E_2}$ lies in between  $E_1 \times E_2$ and  $K_{E_1} \times K_{E_2}$. One has 
$$E_1 \times E_2 \stackrel{\pi_{E_1 \times E_2}}{\longrightarrow} K_{E_1 \times E_2} \stackrel {\psi_{E_1 \times E_2}}{\longrightarrow} K_{E_1} \times K_{E_2}$$
Under the involution $\io_A$, $(P_x,P_{x'})$ and $(P_{-x},P_{-x'})$ are identified and so are the points $(P_x,P_{-x'})$ and $(P_{-x},P_{x'})$. 

Since $K_E  \simeq \CP^1$ the Kummer surface is a double cover of $\CP^1 \times \CP^1$ ramified at 8 lines, the lines $\{P\} \times \CP^1$ and $\CP^1 \times \{P\}$ where $P$ is one of the two torsion points $0,1,\infty$ or $\lb_1$,  in the first case, or $\lb_2$ in the second. The sixteen points of intersection of the 8 lines are the images of the $2$-torsion points. By abuse of language, we will call these points two-torsion points as well. 

Let $L_c$, $c\in \{0,1,\infty,\lb_2\}$ denote the lines $\CP^1 \times \{c\}$ and ${}_cL$, $c \in \{0,1,\infty,\lb_1\}$ denote the lines $\{c\} \times \CP^1$. Let $\#_{\lb_1,\lb_2}$ denote the configuration of the 8 lines. 

\begin{center}
	
\begin{tikzpicture}
	
	\filldraw[black,thick](0,0) circle(2pt) node[below left]{$(\lb_1,\lb_2)$};
	\draw[line width=.5mm] (-1,0)--(6,0);
	\draw [line width=.5mm](0,-1)--(0,6);
	\draw  [thin] (1.5,-1)--(1.5,6);
	\draw  [thin] (3,-1)--(3,6);
	\draw  [thin] (4.5,-1)--(4.5,6);
	\draw  [thin] (-1,1.5)--(6,1.5);
	\draw  [thin] (-1,4.5)--(6,4.5);
	\draw  [thin] (-1,3)--(6,3);
	\filldraw[black,thick](0,1.5) circle(2pt) node[below left ]{$0$};
	\filldraw[black,thick](0,3) circle(2pt) node[below left]{$1$};
	\filldraw[black,thick](0,4.5) circle(2pt) node[below left]{$\infty$};
	\filldraw[black,thick](1.5,0) circle(2pt) node[below left]{$0$};
	\filldraw[black,thick](3,0) circle(2pt) node[below left]{$1$};
	\filldraw[black,thick](4.5,0) circle(2pt) node[below left]{$\infty$};
	
\end{tikzpicture}

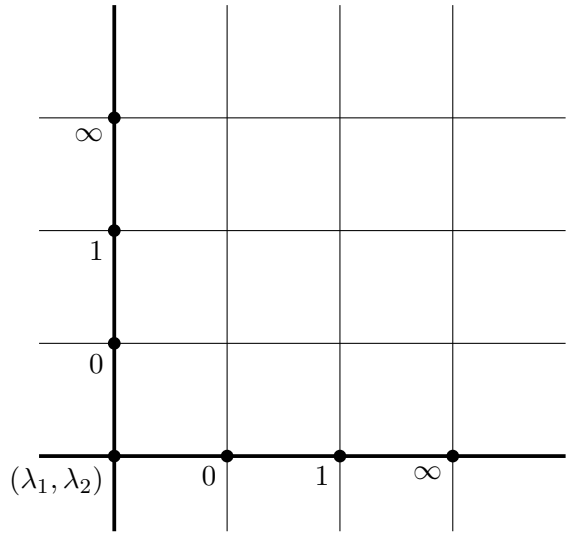
$K_{E_1}\times K_{E_2}$ with the 8 ramified lines. 

\end{center}

The blow up of the sixteen nodes on the Kummer surface is a $K3$ surface and that is sometimes called the Kummer surface. We will call it the {\bf Kummer $K3$ surface} and denote it by $\TK_{E_1\times E_2}$. It is on this surface that we construct the cycles. Note that it is only the highlighted lines $\{\lb_1\} \times \CP^1$ and $\CP^1 \times \{\lb_2\}$ which change as we vary $\lb_1$ and $\lb_2$. 

\section{Motivic Cycles}

In this section we introduce some background on motivic cycles on a surface.

\subsection{Motivic cycles on a surface} 

Let $X$ be a surface. We are interested in constructing elements in the motivic cohomology group $H^3_{\M}(X,\Q(2))$ which is the same as the higher Chow group $\CH^2(X,1)$ (\cite{MVW}, Lecture 19, page 159). An element $\xi \in H^3_{\M}(X, \Q(2))$ can be represented as a sum 
$$\xi=\sum (C_i,f_i)$$
where $C_i$ are irreducible curves on $X$ and $f_i$ a non-zero rational function on $C_i$ such that they satisfy the co-cycle condition 
$$\sum \div_{C_i}(f_i)=0$$ 
as a codimension-2 cycle on $X$.
Relations in this group are given by the tame symbol map $\tau$,
$$K_2(k(X)) \stackrel{\tau}{\longrightarrow} \bigoplus_{C \in X^{(1)}}  k(C)^*$$
defined as follows: If $f$ and $g$ are two non-zero rational functions on $X$, let $\{f, g\}$ denote the Steinberg symbol. By Matsumoto's theorem, the group $K_2(k(X))$ is generated by these symbols subject to bilinearity and the Steinberg relation $\{f, 1-f\} = 1$ for $f \in k(X)^\times \setminus \{1\}$. Define the tame symbol map $\tau$ on Steinberg symbols by 
$$\tau(\{f,g\})= \sum_{C \in  X^{(1)}} \left(C,(-1)^{\ord_C(f)\ord_C(g)}\frac{f^{\ord_C(g)}}{g^{\ord_C(f)}}\right) $$
where the summation runs over the set $X^{(1)}$ of irreducible codimension $1$ closed subvarieties of $X$; extend the map to all of $K_2(k(X))$ by linearity. A cycle in $H^3_{\M}(X,\Q(2))$ is then equivalent to $0$ if it is the image $\tau(\alpha)$ for some $\alpha$ in $K_2(k(X))$.

One way of constructing motivic cycles is the following: 

\begin{prop}\label{prop-higher-chow-cycles-blow-up}  
Let $Q$ be a {\em nodal rational curve} on $X$ with node $P$. Let $\nu:\tilde{Q}\rightarrow Q$ be its strict transform in the blow up $\tilde{X}$ of $X$ at $P$.  $\tilde{Q}$ meets the exceptional fibre $\RE_P$ at two points $P_1$ and $P_2$. Both $\tilde{Q}$ and $\RE_P$ are rational curves. Let $f_P$ be the function with $\div(f_P)=P_1-P_2$ on $\tilde{Q}$ and similarly let $g_P$ be the function with $\div(g_P)=P_2-P_1$ on $\RE_P$. Then 
	$$(\tilde{Q},f_P)+(\RE_P,g_P)$$
	is an element of $H^3_{\M}(\tilde{X},\Q(2))$. 
	\label{construction}
\end{prop}
This was used by Lewis and Chen \cite{chle} to prove the Hodge ${\mathcal D}$-conjecture for certain $K3$ surfaces  and by me \cite{sree2014} to prove an analogue of that in the non-Archimedean case for Abelian surfaces.

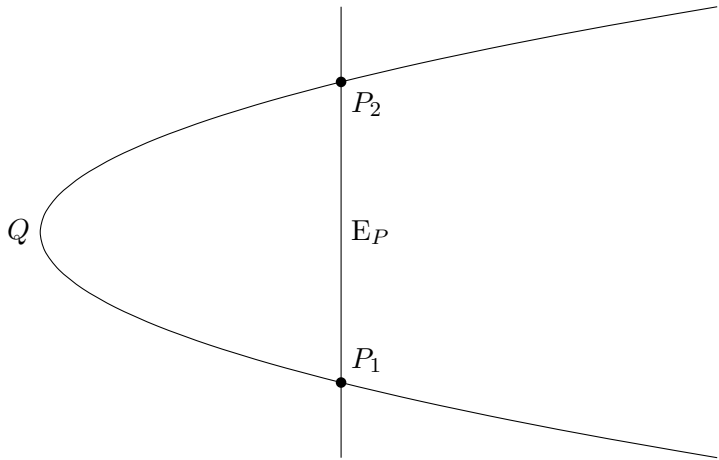
\begin{center}
	
\begin{tikzpicture}
	
	\draw[name path=conic,domain=-3:3,smooth,variable=\x,black] plot ({\x*\x},{\x});
	
    \node[left] at  (0,0) {$Q$};
	
	\draw[name path=line, black] (4,-3) -- (4,3);
	\path[name intersections={of=conic and line,by={A1,A2}}];
	
	\node [right] at (4,0) {$\RE_P$};
	\foreach \point in {A1,A2}
	\fill (\point) circle (2pt);
	\node[above right] at (A1) {$P_1$};
	\node[below right] at (A2) {$P_2$};
	
\end{tikzpicture}

\end{center}

\subsection {Decomposable and Indecomposable Cycles}

\label{indecomp}

In the group $H^3_{\M}(X,\Q(2))$, {\em decomposable cycles} are those in the image of the product map 
$$H^1_{\M}(X,\Q(1)) \otimes H^2_{\M}(X,\Q(1)) \longrightarrow H^3_{\M}(X,\Q(2))$$
The group $H^1_{\M}(X,\Q(1))$  is $K^*$, where $K$ is the field of definition of $X$ and the group $H^2_{\M}(X,\Q(1))$ is $\Pic(X)\otimes \Q$. Over an algebraically closed field, the decomposable cycles are of the form $(C,a)$ where $a$ is a constant function. Let $H^3_{\M}(X,\Q(2))_{dec}$ denote the group of decomposable cycles. 

The non-zero classes in
\[H^3_{\M}(X,\Q(2))_{ind} = H^3_{\M}(X,\Q(2))/H^3_{\M}(X,\Q(2))_{dec}\]
are called \emph{indecomposable cycles}. In general, they are difficult to construct. 
\[\]

Let $\XX \rightarrow S$ be a family of surfaces  over a base $S$ with $X_s$ the fibre over an irreducible codimensional one closed subvariety $s$ of $S$ and $X_{\eta}$ the fibre over the generic point $\eta=Spec(k(S))$. We say a cycle in $CH^1(X_s)_{\Q}=H^2_{\M}(X_s,Q(1))$ is {\em generic} if it is the restriction to $X_s$ of the closure of a cycle in $CH^1(X_{\eta})_{\Q}=H^2_{\M}(X_{\eta},\Q(1))$. We denote the group of generic cycles by $H^2_{\M}(X_s,\Q(1))_{gen}$. 

A motivic cycle in $H^3_{\M}(X_s,\Q(2))$ is said to be {\em generically decomposable} if it lies in the image of 
$$H^1_{\M}(X_s,\Q(1)) \otimes H^2_{\M}(X_s,\Q(1))_{gen} \longrightarrow H^3_{\M}(X_s,\Q(2))$$
and we denote the image by $H^3_{\M}(X_s,\Q(2))_{gen-dec}$. 

A cycle is said to be {\em generically indecomposable} if it is non-zero in quotient group  
$$H^3_{\M}(X_s,\Q(2))_{gen-ind}=H^3_{\M}(X_s,\Q(2))/H^3_{\M}(X,\Q(2))_{gen-dec}$$
Note that when the Chow group of $X_s$ becomes larger a cycle could be decomposable but generically indecomposable. 

In this paper we show that certain cycles we construct using Proposition~\ref{construction} in the Kummer $K3$ surface of a product of two elliptic curves are generically indecomposable. In the companion paper \cite{sree-simple} we show that this is also the case for the Kummer surface of a simple Abelian surface.

\subsection{Regulators}

\label{regulators}

\subsubsection{Archimedean regulator}

Beilinson \cite{beil} defined a regulator map from the motivic cohomology group to the Deligne cohomology group  
$$H^3_{\M}(X,\Q(2)) \stackrel{\reg}{\longrightarrow} H^3_{\D}(X,\Q(2))$$ 
Here $H^3_{\D}(X,\Q(2))$ can be identified with a generalized torus 
$$H^3_{\D}(X,\Q(2))=\frac{F^1H^2(X,\C)^*}{H_2(X,\Q)}$$
Explicitly it can be described as follows: Let $\xi=\sum (C_i,f_i)$ be a cycle and $\omega$ in $F^1H^2(X,\C)$. Let $[\infty,0]$ be the path from $\infty$ to $0$ along the real axis and $\gamma_i=f_i^*([\infty,0])$. Let $\gamma=\sum \gamma_i$. The condition $\sum \div(f_i)=0$ implies that the cycle $\gamma$ is, up to torsion, a boundary; so there is a 2-cycle $\Gamma$ such that $\partial(\Gamma)=\gamma$. 

The regulator is then given by
$$\langle  \reg(\xi),\omega \rangle=\int_{\Gamma} \omega + \frac{1}{2\pi \sqrt {-1}} \sum_i \int_{C_i \backslash \gamma_i}  \log(f_i)\omega$$
where $\log(f_i)$ is pull back of the  the principal branch of the logarithm on $\CP^1 \backslash [\infty,0]$ by $f_i$. 

The Archimedean regulator  or Real regulator is the map to $H^3_{\D}(X,\R(2))$. Since 
$$H^3_{\D}(X,\R(2))=\frac{F^1H^2(X,\C)^*}{H_2(X,\R)}=\frac{F^1H^2(X,\R \oplus \R(1))^*}{H_2(X,\R)}=F^1H^1(X,\R(1))^*$$
so it is a current on $F^1H^1(X,\R(1))$ and becomes, for $\omega \in F^1H^1(X,\R(1))$, 
$$\langle \reg_{\R}(\xi),\omega \rangle=\frac{1}{2\pi \sqrt {-1}}\sum \int_{C_i \backslash \gamma_i} \log|f_i|\omega$$
The advantage of considering this is that the target is a real vector space and not a generalized torus, so determining if the real regulator is non-zero is easier. 

Note that our description descends to a well-defined map on the motivic cohomology group. For a relation in the group coming from the tame symbol of functions the regulator can be seen to be $0$. 

\subsubsection {`Non-Archimedean' regulator}

\label{nonarchregulator}

One also has a `non-Archimedean' regulator which is the boundary map in a certain localization sequence. Let $S$ be a variety, $s$ an irreducible divisor and $U=S \backslash \{s\}$. If $\XX \rightarrow S$ is a family of surfaces over $S$ and $X_U$ and $X_s$ the fibres over $U$ and $s$ respectively, the usual localization sequence relates the Chow groups of $\XX$ and $X_U$ and $X_s$. In terms of motivic cohomology groups (recall $CH^p(X)=H^{2p}_{\M}(X,\Q(p))$) this is 
$$H^2_{\M}(X_s,\Q(1)) \rightarrow H^{4}_{\M}(\XX,\Q(2)) \rightarrow H^{4}(X_U,\Q(2))$$

However, it is not left exact. The localization sequence for motivic cohomology groups \cite{bloc} extends this to a long exact sequence. We can also extend it by replacing $s$ by all irreducible divisors of $S$  and $U$ by the generic point $\eta$  of $S$. Let $X_{\eta}$ denote the generic fibre. Then one has 
$$\cdots \longrightarrow  H^3_{\M}(X_{\eta},\Q(2)) \stackrel{\oplus_{s \in S^1} \partial_s}{\longrightarrow} \bigoplus_{s \in S^1}  H^2_{\M}(X_s,\Q(1)) \longrightarrow H^4_{\M}(\XX,\Q(2)) \longrightarrow \cdots $$
Here $S^1$ is the set of all irreducible closed subvarieties of codimension $1$  of $S$. 

The boundary map $\partial=\oplus \partial_s$ is defined as follows
$$\partial \left(\sum (C_i,f_i)\right)= \sum_i \div_{\bar{C_i}}(f_i).$$
where $\bar{C_i}$ is the closure of $C_i$ in  $\XX$. The co-cycle condition $\sum \div(f_i)=0$ ensures that $\sum \overline{\div(f_i)}=0$ so the image of $\partial$ is supported on the `vertical' fibre $X_s$ since the horizontal components of the divisor cancel out.

We call $\partial_s$ the `non-Archimedean' regulator as in the case when $S$ is $Spec(\OO_K)$, where $\OO_K$ is the ring of integers of a number field, this is boundary map to the fibres over a finite prime. In fact the Archimedean regulator can be viewed as the `boundary map to the fibre at $\infty$' in the sense of Manin \cite{mani}.

\subsubsection{Applications to indecomposability}

We have the following proposition - which is due to Chen and Lewis \cite{chle} in the Archimedean case and me \cite{sree2014} in the non-Archimedean case. Let $\XX \rightarrow S$, $X_{U}$ and $X_s$ as before.  

\begin{prop} Let $\xi_{\eta}$ be a cycle in $H^3_{\M}(X_{\eta},\Q(2))$ and $\xi_s$ the restriction to the fibre over $s$ - if it is defined. Let $NS(X_{\eta}) \subset H^2(X,\Q)$ be the Neron-Severi group of the fibre over $U$ - namely the image of  $CH^1(X_{\eta})=H^2(X_{\eta},\Q(1))$ under the cycle class map. Then if either of the following holds 
\begin{itemize} 
	
	\item (Archimedean) For some  $\omega_{\eta} \in F^1H^2(X_{\eta})$ in $NS(X_{\eta})^{\perp}$ 
$$\langle\reg(\xi_{\eta}),\omega_{\eta}\rangle \neq 0$$

\item (Non-Archimedean) If there is a $Z_s \in CH^1(X_s)$ which is not in $CH^1(X_s)_{gen}$ and further 
$$\langle \partial (\xi_{\eta})|_{X_s},Z_s \rangle \neq 0$$
where $\langle \;,\;\rangle$ is the intersection pairing on $H^2(X_s,\Q)$.
\end{itemize}

then $\xi_U$ is indecomposable. 
\label{indecomposiblecriterion}
\end{prop}

\begin{proof}If $\xi=(C_{\eta},a_{\eta})$ is a decomposable cycle, then the class of $C_{\eta}$ lies in $NS(X_{\eta})$ and 
$$\langle \reg(\xi_{\eta}),\omega_{\eta}\rangle = \log(a_{\eta}) \int_{C_{\eta}} \omega_{\eta} =0$$
since by assumption $\omega_{\eta}$ is in $NS(X_{\eta})^{\perp}$.

Similarly $\partial(\xi)|_s=\ord_s(a) \bar{C}_{\eta} |_s$ and so 
$$(\partial(\xi)|_s,Z_s)=\ord_s(a) (\bar{C_{\eta}}|_s,Z_s)=0$$
where $\bar{C_{\eta}}$ is the closure of $C$ in $\XX$. Hence, in either case, if the pairing is non-zero, the cycle cannot be decomposable. 

\end{proof}

The condition essentially says that if the boundary $\partial(\xi)|_s$ involves `new' cycles which do not come from the restriction of generic cycles then the cycle is generically indecomposable.

A cycle satisfying the  Archimedean condition is usually called {\em regulator indecomposable}. Most authors \cite{chle}, \cite{sato}, \cite{coll} compute 
$$\langle\reg(\xi_{\eta}),\omega_{\eta} \rangle$$
as a function on the base, analogous to a normal function, and show that it is non-trivial, which implies that the cycle is generically indecomposable. The non-Archimedean case is sometimes easier as one might have a good description of the `new' cycles in the special fibre $X_s$ from which it is easy to show that the pairing is non-zero. This is what we do. 

\section{Motivic cycles on the Kummer Surface}

In this section we construct motivic cycles in the group $H^3_{\M}(E_1 \times E_2,\Q(2))$ where $E_1$ and $E_2$ are elliptic curves. The construction is a generalization of the work of Sato \cite{sato}. In \cite{sree-simple} we had a very  similar construction of elements in  $H^3_{\M}(A,\Q(2))$ where $A$ is a simple Abelian surface.

\label{sec-mot-cycles-on-Km-surfaces}

Let $E_1=E_{\lb_1}$ and $E_2=E_{\lb_2}$ be two elliptic curves with full level $2$ structure. Suppose there exists an isogeny of odd degree $n$,  
$$\phi:E_1 \rightarrow E_2$$
Let 
$$\Gamma_{\phi}=\{(x,\phi(x))|x \in E_1\}$$ 
be the graph of the isogeny. This is an elliptic curve in $E_1 \times E_2$.

Let 
$$\pi:E_1 \times E_2 \longrightarrow \KE$$
be the map to the Kummer surface and let 
$$C_{\phi}=\pi(\Gamma_{\phi})$$
Since $n$ is odd, $\phi$ takes the non zero 2-torsion points of $E_1$ to the non-zero 2-torsion points of $E_2$ and $0$ in $E_1$ to $0$ in $E_2$. 

The map $\pi:\Gamma_{\phi} \longrightarrow C_{\phi}$ is a double cover ramified only at the four points $(\alpha,\phi(\alpha))$ on $E_1 \times E_2$,  where $\alpha$ is in $E_1[2]$, the set of 2-torsion points of $E_1$. 

Since $\phi:E_1 \rightarrow E_2$ is an isogeny, $\phi(E_1)=E_2$ and they have the same $j$-invariant. So $\phi(\lambda_1)$ is one of the six pre-images of $j(E_2)$ under $\nu$, one of which is $\lambda_2$. Therefore we  may  assume that $(\lb_1,\lb_2)$ lies on $C_{\phi}$ as the argument is the same in all the cases. 

So without loss of generality, we may assume that the branch points are  $(\lb_1,\lb_2)$, $(0,\phi(0))$, $(1,\phi(1))$ and $(\infty,\phi(\infty))$, where by $\phi(c)$ for a cusp $c$ which corresponds to a two torsion point $P$ on $E_1$ we mean the cusp corresponding to the two torsion point $\phi(P)$. From the Riemann-Hurwitz theorem  $C_{\phi}$ is a {\em rational} curve. 

We also have the map 
$$\psi: \KE\longrightarrow K_{E_1} \times K_{E_2} \simeq \CPP$$
which is ramified at the eight lines $\{c\} \times \CP^1$ and $\CP^1 \times \{c\}$ where $c$ is $\{0,1,\infty\}$ or $\lb_1$ in the first case and $\lb_2$ in the second. Let 
$$Q_{\phi}=\psi(C_{\phi})$$
For any pair of points $(P,P')$ the four points $(P,P'), (P,-P'), (-P,P'), (-P,-P')$ are identified under the map $\psi \circ \pi$. So both  $\Gamma_{\phi}$ and $\Gamma_{-\phi}$ map to the same curve $Q_{\phi}$ in $\CPP$. However, in $\KE$, the curves $C_{\phi}$ and $C_{-\phi}$ are not the same though
$$Q_{\phi}=\psi(C_{-\phi})$$
as well. 
\subsection{Geometry of $\CPP$}

The Neron-Severi group of $\CPP$ is isomorphic to $\ZZ\oplus \ZZ$ and is  generated by the rulings $\CP^1 \times \{P\}$ and $\{Q\} \times \CP^1$ for any points $P$ and $Q$.  Since $H^1(\CPP,\ZZ)=0$, any curve  $C$ in $\CPP$ is rationally equivalent to a curve of the form $a(\CP^1\times \{P\}) + b(\{Q\} \times \CP^1)$ where $a=(C,\CP^1 \times \{P\})$ and $b=(C,\{Q\} \times \CP^1)$ are the intersection numbers of the two curves. The tuple $(a,b)$ is called the {\em bidegree} of $C$. 

Suppose $\phi:E_1 \rightarrow E_2$ is an isogeny of odd degree $n$.  
Hence there are $n$ points of the form $(x,0)$ on $E_1\times 0$ and the one point $((0,0)$ on $0 \times E_2$ on $\Gamma_{\phi}$. These points map to distinct points on $Q_{\phi} \cap (\CP^1 \times \lb_2)$ and $Q_{\phi} \cap (\lb_1 \times \CP^1)$ and every point of intersection is of this type. Therefore the curve $Q_{\phi}$ is curve of bidegree $(n,1)$. 

We have the following lemma about curves in $\CPP$.

\begin{lem} Let $C$ be a curve on $\CPP$  of bidegree $(a,b)$. 
	\label{coray}. 
	
	\vspace{\baselineskip}
	
\noindent (1). The arithmetic genus of the curve $C$ is 
	$$p_a(C)=(a-1)(b-1)$$ 
	
\noindent (2). The dimension of the linear system of the divisor $C$ in $\CPP$ is $$\dim|C|=ab+a+b=2(a+b)+p_a-1$$. 
	
\end{lem}

\begin{proof} Coray \cite{cora}, Lemma 1.1 
	\end{proof}
In particular, $p_a(Q_{\phi})=0$, so it is a {\em rational} curve. Further, the dimension of the linear system $|Q_{\phi}|$ is $2n+1$. 

Recall that we have the eight distinguished lines $L_c$ and ${}_cL$ in $K_{E_1}\times K_{E_2}$ which are ramified under the map $\psi: \KE \longrightarrow K_{E_1} \times K_{E_2}$. In general, a curve of bidegree $(a,b)$ will meet the $8$ lines at $4(a+b)$ points so for $Q_{\phi}$ the expected number is $4(n+1)$.

However, when $\lambda_1=\lambda_2$, observe that the diagonal $Q_{id}$ of bidegree $(1,1)$ meets the $8$ lines only at the $4$  points $(\lb_1,\lb_2)$, $(0,0)$, $(1,1)$ and $(\infty,\infty)$ and {\em not} the expected $8$ points, since all the points of intersection are double points. We show that this is also the case in general.

\begin{lem} Every point of intersection of $Q_{\phi}$ with the  eight distinguished lines $L_c$ and ${}_cL$ has multiplicity $2$. Hence there are only  $2(n+1)$ distinct points. 
	\label{doubleintersection} 
\end{lem}

\begin{proof}
	
	Let $L_c$ and ${}_cL$ denote the eight distinguished ramified  lines and $\#_{\lb_1,\lb_2}$ the configuration of the eight lines.

	 We will show that the set of points in $\KE$ lying over  $Q_{\phi} \cap  \#_{\lb_1,\lb_2}$  is  the set 
	$C_{\phi}\cap C_{-\phi}$. Since $Q_{\phi}$ is the image of both $C_{\phi}$ and $C_{-\phi}$ the image of the points of $C_{\phi} \cap C_{-\phi}$ intersect the lines $L_c$ and ${}_cL$  with  multiplicity $2$. Further, they are nodes of $\psi^{-1}(Q_{\phi})$.
	
	Let $P$ be a point in $C_{\phi} \cap C_{-\phi}$. Then $P$ lies in the image of $\Gamma_{\phi} \cap \Gamma_{-\phi}$. 
	$$\Gamma_{\phi} \cap \Gamma_{-\phi}=\{(x,y)| \phi(x)=y=-\phi(x)=\phi(-x)\}.$$
	This implies either $x=-x$ or $\phi(x)=-\phi(x)$. Therefore at least one of $x$ or $\phi(x)$ is a 2-torsion point. Conversely if $(x,y)$ is on $\Gamma_{\phi}$ where $x$ or $y$ is a 2- torsion point then 
	$$-\phi(x)=\phi(-x)=\phi(x)  \text{ if } x \in E_1[2]$$
	or  
	$$\phi(x)=y=-y=-\phi(x) \text{ if }  y \in E_2[2]$$
	So this implies that 
	$$\Gamma_{\phi} \cap \Gamma_{-\phi} =\Gamma_{\phi} \cap \left( \bigcup_{x\in E_1[2]}  \{x\} \times E_2 \cup \bigcup_{y \in E_2[2]} E_1 \times \{y\}\right)$$
	The image of these sets under $\pi$ imply 
    $$C_{\phi} \cap C_{-\phi}=C_{\phi} \cap \left(\bigcup_{c \in \{0,1,\infty,\lb_2\}} L_c \cup \bigcup_{c \in \{0,1,\infty,\lb_1\}} {}_cL\right)$$
    and finally under $\psi$ this maps to 
    $$Q_{\phi} \cap \left(\bigcup_{c \in \{0,1,\infty,\lb_2\}} L_c \cup \bigcup_{c \in \{0,1,\infty,\lb_1\}} {}_cL\right)$$
	
\end{proof}

The points are of the following type - there are $4$ points of the type $(c,\phi(c))$ where $c$ is a cusp and $\phi(c)$ is as before. The others are the images of the points where the curve $\Gamma_{\phi}$ meets $E_1 \times \{P_c\}$ where $P_c$ is a point in $E_2[2]$ corresponding to the point $c$. Since $\phi$ is of odd degree $n$ there are $(n-1)$ such points $(x,P_c)$ with $x \notin E_1[2]$.
 
However, if $\phi(x)=P_c$ so is $\phi(-x)$ and $-x \neq x$. But $(-x,P_c)$ and $(x,P_c)$ map to the same point on $Q_{\phi} \cap  L_c$. Hence there are only $(n-1)/2$ points for each $c$. Since there are $4$ points $c$ we have $Q_{\phi} \cap \#_{\lb_1,\lb_2}$ has  $4\frac{(n-1)}{2}+4=2n+2$ distinct points. 

\subsection{Enumerative geometry of rational curves in  $\CPP$}
\label{ssec-enumerative-geom-ratl-curves-p1-p1} 

If $E_1$ and $E_2$ are isogenous by an isogeny $\phi$ of degree $n$ we see  there is a rational curve $Q_{\phi}$ of bidegree $(n,1)$ which meets the eight distinguished lines at $2(n+1)$  points - and  all of them are double points. Four are double points of the configuration $\#$ and each line $L_c$ is tangential to $Q_{\phi}$ at $\frac{(n-1)}{2}$ points. 

We would like to `deform' this curve to a curve which is defined for all $E_1$ and $E_2$ in the sense that we would like to find a family of rational curves of bidegree $(n,1)$ in the universal family of configurations of lines in $\CPP$ meeting the  lines $L_c$ and ${}_cL$ at many points and further, coincides with $Q_{\phi}$ when $E_1$ and $E_2$ are isogenous via $\phi$. 

While we cannot do this for all the $2n+2$ points, a theorem in the enumerative geometry of $\CPP$, allows us to find such a rational curve if we discard one of the double points. 

Let $\Sigma=X(2)\times X(2)=\CPP$. For $s=(\lb_1,\lb_2) \in \Sigma$ let $\#_s$ denote the $(4,4)$ configuration in $(\CPP)_s=K_{E_{\lb_1}} \times K_{E_{\lb_2}}$. 

Let $T=T^{\gamma}_n$ denote a component of the Hecke correspondence $T_n \subset \Sigma$.  There is a generic  isogeny $\phi_{\eta_T}$ of odd degree $n$ from $E_{1,\eta_T} \rightarrow E_{2,\eta_T}$ where $\eta_T$ is the generic point of  $T$. If $s$ is a point on $T$ let $\phi_s=\phi_{\eta_T}|_s$. 

\begin{thm} Let $n$ be an odd number. Let $\M_{(n,1)}$ denote the  space of (rational) curves of bidegree $(n,1)$ in $\CPP$. Consider pairs $(Q_{(n,1),s},\#_s) \in \M_{n,1} \times \Sigma$   In general $\langle Q_{n,1,s},\#_s\rangle=4n+4$. 
  
 We say  $(\#_s,Q_{(n,1),s})$ satisfies {\bf Condition A} if $Q_{n,1,s} \cap \#_s$ consists of 
	
	\begin{itemize}
		\item  $3$  double points, say $p_1(s), p_2(s)$ and $p_3(s)$  of the configuration $\#_{s}$, no two of which are on the same ruling. 
		\item $L_c(s)$ is tangential to $Q_{n,1,s}$ at $\frac{n-1}{2}$ points for each $c(s)$. 
		\item $2$ other points (which could coincide). 
	\end{itemize}
Let $I \subset \M_{(n,1)} \times \Sigma$ be the incidence scheme consisting of pairs $(Q_{n,1,s},\#_s)$ satisfying {\bf Condition A} and let 
$$\pi:I \rightarrow \Sigma$$ 
$$\pi((Q_{n,1,s},\#_s))=s$$
Suppose $\pi^{-1}(s_0)$ is non-empty and finite for some $s_0$. Then $\pi^{-1}(s)$ is non-empty  and finite  for all $s \in \Sigma$.

From Lemma \ref{doubleintersection}, the curve $Q_{n,1,s_0}=Q_{\phi_{s_0}}$ satisfies {\bf Condition A} for $s_0\in T^{\gamma}_n$. Hence for any $s\in \Sigma$ there  exists a rational curve of bidegree $(n,1)$ satisfying {\bf Condition A}

\label{enumgeom}
\end{thm}
\begin{proof}
This is an analogue of the statement that there are finitely many rational curves of degree $d$ passing through $k$ points and tangent to $3d-1-k$ lines in $\CP^2$. 

Let $\M_{(n,1)}$ be the space of rational curves of bi-degree $(n,1)$ in $\CPP$. We know this is $(2n+1)$ dimensional. The condition that $Q_{n,1}$ passes through a point reduces the dimension by $1$ and the condition that $Q_{n,1}$ is tangent to a line also reduces the dimension by $1$. These $(2n+1)$ conditions are transversal and cut out $I$, hence $I$ is an l.c.i and  
$$\dim(I)=\dim(M_{(n,1)}\times \Sigma) -(2n+1)=2n+1+2-(2n+1)=2=\dim(\Sigma)$$ 
and the map $\pi$ has finite fibres. Since $I$ is an l.c.i it is  Cohen-Macaulay. By the local criterion for flatness (miracle flatness theorem)  since $I$ is Cohen-Macaulay and $\Sigma$ is smooth, the map $\pi$ is flat. 

Since $\pi$ is flat, from Hartshorne \cite{hart}, Theorem 9.9, we have that the Hilbert polynomial is constant. At $s_0$ we know the Hilbert polynomial  is  length of the fibre which is non-zero.  Hence that is the case in  every fibre. 

Explicitly, the equation $aZP(X,Y)+bWQ(X,Y)$ has $(2n+2)$ coefficients. The condition that the curve passes through a point or is tangent to a line determines a linear relation among the coefficients. So we have $2n+1$ linear relations which on a $2n+2$ dimensional space. Since we know there is one solution the system is not inconsistent.

\end{proof}

Coray \cite{cora} shows that there is a unique curve of bidegree $(n,1)$ passing through  $2n+1$ points in  general position in $\CPP$. There is perhaps a similar theorem stating that there are finitely many  curves of bidegree $(n,1)$ passing through $k$ points and tangent $2n+1-k$  lines. 

We have the following corollary

\begin{cor} There is a family of curves $Q^{\gamma}_{n,1,s}$ which restricts to $Q_{\phi_s}$ when $s$ lies on $T^{\gamma}_n$. 
	\label{deformationofQphi} 
\end{cor}

\begin{proof} Let $I_0$ be an irreducible component of $I$ containing the `seed' $Q_{(n,1),s_0}=Q_{\phi_{s_0}}$. Since $\pi: I \rightarrow \Sigma$ is finite and flat, $I_0$ dominates $\Sigma$. Hence the curve corresponding to the generic point of $I_0$ is a generic curve and by construction restricts to $Q_{\phi_s}$ for $s \in T_n^{\gamma}$. Let  $Q^{\gamma}_{(n,1),s}$ be that curve.

\end{proof}

We have the curves $Q^{\gamma}_{n,1,s}$ in $K_{E_{\lb_1}}\times K_{E_{\lb_2}}$ for every $s=(\lb_1,\lb_2)$. We now use them to construct  rational curves in $K_{E_{\lb_1}\times E_{\lb_2}}$ for {\em any} point $(\lb_1,\lb_2)$ on $X(2) \times X(2)$.

\begin{prop} Let $T=T^{\gamma}_n$ denote a component of the Hecke correspondence in $X(2) \times X(2)$ and $\phi_{\eta_T}$ a generic  isogeny of degree $n$ as above. 
	
Let $A_{\eta}=(E_{1} \times E_{2})_{\eta}$ be the generic product of elliptic curves over the generic point $\eta$ of $\Sigma=X(2) \times X(2)$. Then there  exists a rational curve $C^{\phi,\gamma}_{n,1,\eta}$ in the generic fibre $K_{A_{\eta}}$ such that 

	\begin{itemize}
		\item If $s \in \Sigma \backslash T$ then  $C^{\phi,\gamma}_{n,1,s}$ is {\bf irreducible}
		\item If $s \in T$  then $C^{\phi,\gamma}_{n,1,s}$ is {\bf reducible} and  
		$$C^{\phi,\gamma}_{n,1,s}=C_{\phi,s} \cup C_{-\phi,s}$$
		\end{itemize}
Here 	$C^{\phi,\gamma}_{n,1,s}=C^{\phi,\gamma}_{n,1,\eta}|_{K_{A_s}}	$
		\label{rationalcurve}
\end{prop}
\begin{proof} From Theorem \ref{enumgeom} we have a curve $Q_{n,1,s} \subset (\CPP)_s$ of bidegree $(n,1)$ for every $s\in\Sigma$. 
Let $C^{\phi,\gamma}_{n,1,s}=C_{n,1,s}$ be the double cover of $Q_{n,1,s}$ 

If $s$ lies on $T$ then by construction the curve is $Q_{\phi_s}$ which meets $\#_s$ at $2n+2$ double points. In $C_{n,1,s}$ these are the ramification points and hence {\em all the ramification points are singular}.  The  normalization then  is an unramified double cover of $\CP^1$ and hence is the union of two copies of $\CP^1$. From Lemma \ref{doubleintersection} one can see that the two components are $C_{\phi}$ and $C_{-\phi}$. Hence if $s$ lies on $T$, $C_{n,1,s}=C_{\phi} \cup C_{-\phi}$

If $s \in \Sigma\backslash T$ then on $C_{n,1,s}$ there are two ramification points which are smooth. From the Riemann-Hurwitz theorem applied to the normalization one can see that the normalization is an irreducible rational curve. Hence $C_{n,1,s}$ is an irreducible rational curve.

\end{proof}
For instance, if $n=1$, the curve $Q^{id}_{1,1}$ is simply the diagonal in $\CPP$. It passes through $(0,0),(1,1)$ and $(\infty,\infty)$ but does not pass through $(\lb_1,\lb_2)$ unless $E_1$ and $E_2$  are isomorphic as elliptic curves with level $2$ structure. The rational curves $C^{\gamma}_{1,1,s}$ for $s$ not in the diagonal  were used by Sato \cite{sato} to construct motivic cycles. In the next section we show that this can be generalized to our case.

\begin{center}
	
	\begin{tikzpicture}
		
		\filldraw[black,thick](0,0.5) circle(1.5pt) node[below left]{$(\lb_1,\lb_2)$};
		\draw[thin] (-1,0.5)--(6,0.5);
		\draw [thin](0,-1)--(0,6);
		\draw  [thin] (1.5,-1)--(1.5,6);
		\draw  [thin] (3,-1)--(3,6);
		\draw  [thin] (4.5,-1)--(4.5,6);
		\draw  [thin] (-1,1.5)--(6,1.5);
		\draw  [thin] (-1,4.5)--(6,4.5);
		\draw  [thin] (-1,3)--(6,3);
		\filldraw[black](0,1.5) circle(1.5pt) node[below left ]{$0$};
		\filldraw[black](0,3) circle(1.5pt) node[below left]{$1$};
		\filldraw[black,thick](0,4.5) circle(1.5pt) node[below left]{$\infty$};
		\filldraw[black,thick](1.5,0.5) circle(1.5pt) node[below left]{$0$};
		\filldraw[black,thick](3,0.5) circle(1.5pt) node[below left]{$1$};
		\filldraw[black,thick](4.5,0.5) circle(1.5pt) node[below left]{$\infty$};
		\draw [thick] (-1,-1)--(6,6);
		\filldraw[black,thick](0,0) circle(2pt) node[below right]{$t$};
		\filldraw[black,thick](0.5,0.5) circle(2pt) node[above left]{$s$};
		
		\filldraw[black,thick](1.5,1.5) circle(2pt) node[below right]{$(0,0)$};
		\filldraw[black,thick](3,3) circle(2pt) node[below right]{$(1,1)$};
		\filldraw[black,thick](4.5,4.5) circle(2pt) node[below right]{$(\infty,\infty)$};
		\node [above right] at (1.5,2.25) {$Q^{id}_{1,1}$};

	\end{tikzpicture}

The curve $Q_{1,1}^{id}$
	
\end{center}

\subsection{The motivic cycles} \label{motiviccycleconstruction}

Let $n$ be an odd integer and $T=T_n^{\gamma}$ be a component of the Hecke correspondence in $\Sigma=X(2) \times X(2)$. From Proposition \ref{rationalcurve} we have a nodal rational curve  $C_{n,1,\eta}$ on $K_{A_{\eta}}$. Using Proposition \ref{construction} we  obtain a motivic cycle in $H^3_{\M}(\TKE,\Q(2))$. In the case when $\phi$ is the isogeny of degree $1$, this is due to Sato \cite{sato}. 

\begin{lem} \label{lem-higher-chow-cycles-on-KmK3}
Let $A_{\eta}=E_{1,\eta}\times E_{2,\eta}$ be the generic split abelian surface over $\Sigma$. The curve $C_{n,1,\eta}= C^{\phi,\gamma}_{n,1,\eta}$ determines a generically irreducible nodal rational  curve on $K_{A_{\eta}}$. 

If $c_{\eta}$ is a node of $K_{A_{\eta}}$ lying on $C_{n,1,\eta}$  then there exists a cycle 
	$$\xi^{\phi,\gamma}_{n,1,c_{\eta}}=\xi^{\phi}_{c_{\eta}} \in H^3_{\M}(\TK_{A_{\eta}},\Q(2))$$   
	where $\TK_{A_{\eta}}$ is the K3 surface  obtained by blowing up all the nodes of $K_{A_{\eta}}$. 
	
\end{lem}

\begin{proof}

Let $c_{\eta}$ be a node of $C_{n,1,\eta}$ which is also a node of $K_{A_{\eta}}$. From  Theorem \ref{enumgeom} we know there are three such nodes. 

Let $\TC_{n,1,\eta}$ denote the strict transform of $C_{n,1,\eta}$ and let $\RE_{c_{\eta}}$ denote the exceptional fibre over $c_{\eta}$ in the Kummer $K3$ surface $\TK_{A_{\eta}}$ 

Let $P^1_{c_{\eta}}$ and $P^2_{c_{\eta}}$ be the two points in $\TC_{n,1,\eta} \cap \RE_{c_{\eta}}$ lying over the node $c_{\eta}$ and let  $f_{c_{\eta}}$ and $g_{c_{\eta}}$  be functions on $\TC_{n,1,\eta}$ and $\RE_{c_{\eta}}$ respectively with divisors 
$$\div_{\TC_{n,1,\eta}}(f_{c_{\eta}})=P^1_{c_{\eta}}-P^2_{c_{\eta}} \hspace{1in} \div_{\RE_{c_{\eta}}}(g_{c_{\eta}})=P^2_{c_{\eta}} - P^1_{c_{\eta}}$$ 
we further normalize $f_{c_{\eta}}$ by requiring 
$$f_{c_{\eta}}(t_{\eta})=1$$
where $t_{\eta}=(\lb_{1,\eta},y_{\eta})$ is one the two smooth branch points. 

The curve $C_{n,1,\eta}|_s$ is irreducible as long as $s$ does not lie on $T$. In that case using Proposition \ref{construction} we get a motivic cycle in $H^3_{\M}(\TKE,\Q(2))$. However monodromy around $T$ can interchange $P^1_{c_{\eta}}$ and $P^2_{c_{\eta}}$ but by going to an etale cover  $\pi:\Sigma' \rightarrow \Sigma$ we can get a well defined cycle in the generic fibre $\TK_{A_{\eta'}}$, where $\eta'$ is the generic point of $\Sigma'$.  Tracing down gives us a cycle 
$$\xi^{\phi}_{c_{\eta}}=\xi^{\phi,\gamma}_{n,1,c_{\eta}}=\pi_*((\TC_{n,1,\eta'},f_{c_{\eta'}})+(\RE_{c_{\eta'}},g_{c_{\eta'}}))$$
 in the group $H^3_{\M}(\TK_{A_{\eta}},\Q(2))$. 
 
 \end{proof}
 
 The dependence on $\phi$ is because the curve $C_{n,1}$ is obtained by deforming the graph of $\phi$ and the dependence on $c_{\eta}$ is because a different choice of node could give a possibly different cycle. 

In the next section we will show that these cycles are generically indecomposable by computing their non-Archimedean regulator. Further, the cycles corresponding to the six different $\gamma$ are linearly independent as their non-Archimedean regulators are non-zero and different. 

However it appears that $\xi_{n,c_{\eta}}^{\phi,\gamma}$ and $\xi_{n,c'_{\eta}}^{\phi,\gamma}$ have the same non-Archimedean regulator. Sato \cite{sato} shows that when $n=1$ the Archimedean regulators do depend on the choice of cusp $c$ as well - so the space spanned by them is $18$ dimensional. We expect something similar to hold in this case as well.  

For instance, if $n=1$ one has the cycles $\xi^{id,\gamma}_{1,(0,0)}$,$\xi^{id,\gamma}_{1,(1,1)}$ and $\xi^{id,\gamma}_{1,(\infty,\infty)}$ defined on the complement of the diagonal $(\lb,\lb)$. One also has similar  cycles defined on the complement of the lines $(\lb,1/\lb), (\lb,1-\lb)), (\lb,1/(1-\lb)),(\lb,\lb/(1-\lb)$ and $(\lb,(\lb-1)/\lb)$.

\subsection{Indecomposability of the motivic cycle}

\label{indecomposabilityofthecycle}

Recall (Section \ref{nonarchregulator}) that we have the localization sequence for motivic cohomology. We apply it to our case of a family of 
$K3$ surfaces $\TK_{A_s}$ for $s \in \Sigma=X(2) \times X(2)$. The relevant part for us is  
$$ \cdots \rightarrow H^3_{\M}(\TK_{A_{\eta}},\Q(2)) \stackrel{\bigoplus \partial_S}{\longrightarrow} \bigoplus_{S\in \Sigma^1} H^2_{\M}(\TK_{A_S},\Q(1)) \rightarrow H^{4}_{\M}(\TK_{\A_{\Sigma}},\Q(2))\rightarrow \cdots$$
where $\A_{\Sigma}$ is the canonical compactification of $A_{\eta}$. 

To show that our cycle is indecomposable we compute the image of our  cycle $\xi^{\phi}_{c_\eta}$ under the boundary homomorphism and apply Proposition \ref{indecomposiblecriterion}.

\begin{thm} The boundary of the cycle $\xi^{\phi}_{c_{\eta}}$ under the map $\partial=\oplus \partial_S$ is supported in the fibre over $T=T_n^{\gamma}$ and is 
		$$\partial (\xi^{\phi,\gamma}_{n,c_{\eta}})=(\TC_{\phi_T} -\TC_{-\phi_T})$$
up to the boundary of a decomposable element. Here $\TC_{\phi_T}$ is the closure in $\TK_{A_T}$ of  the strict transform  in $\TK_{A_{\eta_T}}$ of the  image $C_{\phi_{\eta_T}}$ of the graph $\Gamma_{\phi_{\eta_T}}$ in $A_{\eta_T}$.

In particular, the cycle $\TC_{\phi_T} -\TC_{-\phi_T}$ is not the restriction of the closure of a  cycle in the fibre over the complement and  hence  $\xi_{c_{\eta}}^{\phi}$ is generically {\em indecomposable}.
\label{mainthm}
	
\end{thm}

\begin{proof} Under the map $\TK_{A_{\eta}}\longrightarrow K_{A_{\eta}}$ all the nodes of $K_{A_{\eta}}$ are blown up. For $c_{\eta}$ a node of $K_{A_{\eta}}$, let ${\mathrm E}_{c_{\eta}}$ denote the exceptional fibre over $c_{\eta}$.
	
To compute $\partial$ we first need to compute the closure of $\TC_{n,1,\eta}$ in the universal family. When $s$ does not lie on $T$ the curve $\TC_{n,1,s}$ is irreducible. 

However, on $\eta_T$, the generic point of $T$, the generically smooth branch points coincide to give another node of $C_{n,1,\eta_T}$ at a point $p_{\eta_T}=(\lb_{1,\eta_T},\lb_{2,\eta_T})$ on $(\CPP)_{\eta_T}$ (as we have chosen the component of $T_n^{\gamma}$ so that $(\lb_1,\lb_2)$ lie on $C_{\phi}$).  $p_{\eta_T}$ is a node of the Kummer surface $K_{A_{\eta_T}}$ as well. Let $p_T$ denote its closure in $K_{A_T}$.  

Note that both the moduli of products of elliptic curves and the product of Kummer surfaces are $\CPP$ and we are making use of the point $(\lb_1,\lb_2)$ on both - one as a moduli point and one as a node of $\#_{\lb_1,\lb_2}$ when $\eta_T=(\lb_1,\lb_2)$ lies on $T$. 

In the  Kummer $K3$ surface $\TK_{A_T}$, the closure $\TCC_{n,1,\eta}$  meets the exceptional fibre $\RE_{p_T}$ over the node $p_T$ at two points and has fibre 
$$\TC_{\phi_T} \cup \TC_{-\phi_T} + \RE_{p_T} $$
where $\TC_{\phi_{\eta_T}}$ and $\TC_{-\phi_{\eta_T}}$ are the strict transforms of the curves $C_{\phi_{\eta_T}}$ and $C_{-\phi_{\eta_T}}$ respectively.

\begin{center}
	
	\begin{tikzpicture}
		
		\draw[name path=conic1,
		domain=1:3,smooth,variable=\x,black] plot ({\x*\x},{\x});
		\draw[name path=conic2,domain=-3:-1,smooth,variable=\x,black] plot ({\x*\x},{\x});

		\draw[name path=line, black] (4,-3) -- (4,3);
		\path[name intersections={of=conic1 and line,by={A1}}];
		\path[name intersections={of=conic2 and line,by={B1}}];
		\node [right] at (4,0) {$\RE_p$};
		\foreach \point in {A1,B1}
		\fill (\point) circle (2pt);
		\node[above right] at (A1) {$\TC_{\phi}$};
		\node[below right] at (B1) {$\TC_{-\phi}$};
		
	\end{tikzpicture}

The fibre over the nodal point $p$ when $(\lb_1,\lb_2)$ lies on $T_n^{\gamma}$.

\end{center}
	
There is an involution $\iota$ on $C_{n,1,\eta}$ determined by the double cover $\KE \longrightarrow \CPP$ which fixes the ramified points. Over $T$ it interchanges the two components. 
$$\iota(C_{\pm\phi_T})=C_{\mp\phi_T}$$
This involution lifts to the  surface $\TK_{\eta}$ and stabilizes the exceptional fibres $\RE_{c_{\eta}}$
$$ \iota(\RE_{c_{\eta}})=\RE_{c_{\eta}}$$
If $P^1_{c_{\eta}}$ and $P^2_{c_{\eta}}$ are the two points lying over a node $c_{\eta}$ then $\iota(P^1_{c_{\eta}})=P^2_{c_{\eta}}$ and $\iota(P^2_{c_{\eta}})=P^1_{c_{\eta}}$.

Let $f^{\iota}_{c_{\eta}}(z)=f_{c_{\eta}}(\iota(z))$. Then 
$$\div_{\TC_{n,1,\eta}}(f^{\iota}_{c_{\eta}})=P^2_{c_{\eta}}-P^1_{c_{\eta}}$$
Therefore $f^{\iota}_{c_{\eta}}=\mu/f_{c_{\eta}}$ for some function $\mu$ on the base. However, we have also normalized $f_{c_{\eta}}$ by requiring  $f_{c_{\eta}}(r_{\eta})=1$, where $r_{\eta}$ is one of the smooth branch points.   

Since $\iota(r_{\eta})=r_{\eta}$, $f^{\iota}_{c_{\eta}}(r_{\eta})=1$ as well. Therefore the function $\mu \equiv 1$. 

Now consider the boundary $\partial(\xi^{\phi}_{c_{\eta}})$. Since outside $T$ the curve  $C_{n,1}$ is irreducible the divisor of $f_c$ is of the form 
$$	\div_{\TCC_{n,1}} (f_c)= \UH+a_{\phi} \TC_{\phi_T}+ a_{-\phi} \TC_{-\phi_T} +  a_p \RE_{p_T} $$ 
where $\UH$ is the `horizontal divisor' -  the closure of $P^1_{c_{\eta}}-P^2_{c_{\eta}}$. Since $f^{\iota}_{c_{\eta}}=1/f_{c_{\eta}}$ one has 
$$\div_{\TCC_{n,1,\eta}} (f^{\iota}_{c_{\eta}})= -\UH - a_{\phi} \TC_{\phi_T} - a_{-\phi} \TC_{-\phi_T} -  a_p \RE_{p_T}$$
On the other hand, $\iota(\div_{\TCC_{n,1,\eta}}(f_{c_{\eta}}))=\div_{\TCC_{n,1,\eta}}(f^{\iota}_{c_{\eta}})$. Further $\iota(\TC_{\pm \phi_T})=\TC_{\mp \phi_T}$, $\iota(\UH)=-\UH$ and $\iota(\RE_{p_T})=\RE_{p_T}$. This implies 
$$\div_{\TCC_{n,1,\eta}}(f^{\iota}_{c_{\eta}})=\iota(\div_{\TCC_{n,1,\eta}}(f_{c_{\eta}}))=-\UH+a_{\phi} \TC_{-\phi_T} + a_{-\phi} \TC_{\phi_T} + a_{p} \RE_{p}$$
Comparing coefficients gives $a_{\phi}=-a_{-\phi}$ and $a_p=-a_p$. Therefore $a_p=0$ and 
$$\div_{\TCC_{n,1,\eta}} (f_{c_{\eta}})= \UH+ a_{\phi}( \TC_{\phi_T} - \TC_{-\phi_T}).$$
We now claim $a_{\phi} \neq 0$. To see this, suppose $a_{\phi}=0$. Then, since $\div(f_{c_{\eta}})$ does not contain $\TC_{\phi_T}$, one has 
$$\deg(\div(f_{c_{\eta}}|_{\TC_{\phi_T}}))=0$$
On the other hand  
$$\deg(\div(f_{c_{\eta}}|_{\TC_{\phi_T}}))=(\div(f_{c_{\eta}}). \TC_{\phi_T})=1$$
so we have a contradiction and $a_{\phi} \neq 0$

The remaining part is $\div_{\bar{E_{c_{\eta}}}}(g_{c_{\eta}})=-\UH$. Hence we have 
$$\partial (\xi^{\phi}_{c_{\eta}})=\partial( ((\TC_{n,1,\eta},f_{c_{\eta}})+(E_{c_{\eta}},g_{c_{\eta}}))=a_{\phi}(\TC_{\phi_T}-\TC_{-\phi_T}).$$
In particular, since this is non-zero, the cycle $\xi^{\phi,\gamma}_{n,c}$ is indecomposable. 
 
To compute $a_{\phi}$ we choose a decomposable cycle $(C_{n,1,\eta},b)$ such that $T_n$ lies in the support of $\div(b)$. Such a $b$ exists as $T_n$ is of bidegree $(k,k)$ for some $k$ so $T_n-k(\{\infty\} \times \CP^1+\CP^1 \times \{\infty \})$ is homologically equivalent to $0$. As $\CPP$ is simply connected, this is also rationally equivalent to $0$. Hence there is a function $b$ with 
$$\div(b)=T_n-k(\{\infty\} \times \CP^1+\CP^1 \times \{\infty \})$$
The boundary of this cycle is 
	$$\partial ((C_{n,1,\eta},b))=\div_{\TCC_{n,1,\eta}} (b)=\TC_{\phi_T}+\TC_{-\phi_T}+ \RE_{p_T} + \text { other terms }$$
	where `other terms' refers to the cycles supported over $\{\infty\} \times \CP^1$ and $\CP^1 \times \{\infty \}$. Therefore 
	$$\div_{\TCC_{n,1,\eta}} (b^{-a_{\phi}}f_{c_{\eta}})=\UH - 2 a_{\phi}\TC_{-\phi_T}- a_{\phi} \RE_{p_T} + \text{ other terms }$$ 
	We can restrict this to $\TC_{\phi_T}$. Since $\TC_{\phi_T}$ does not lie in the support, we have 
	$$\deg(\div b^{-a_{\phi}}f_{c_{\eta}}|_{\TC_{\phi_T}})=0.$$
	On the other hand 
	$$\deg(\div b^{-a_{\phi}}f_{c_{\eta}}|_{\TC_{\phi_T}}) = (\UH,\TC_{\phi_T}) -2a_{\phi}(C_{-\phi_T},C_{\phi_T})- a_{\phi}(\RE_{p_T},\TC_{{\phi_T}}) = 1-a_{\phi}$$
	as $C_{\phi_T}$ does not intersect the `other terms'. Therefore $a_{\phi}=1$.
\end{proof}

In the proof we make the assumption that $P^1_{c_{\eta}}$ lies on $\TC_{\phi_T}$ over $T$. This affects the sign. One has eighteen cycles $\xi_{c_{\eta}}^{\phi}$ coming from the three choices of $c_{\eta}$ and six choices of $\gamma$.

From the calculation above the curves $\xi_{n,c_{\eta}}^{\phi,\gamma}$ have the same boundary for $c_{\eta} \in \{0,1,\infty\}$ but if $\gamma \neq \gamma'$ the cycles  $\xi_{n,c_{\eta}}^{\phi,\gamma}$ and $\xi_{n,c_{\eta}'}^{\phi,\gamma'}$ are distinct. Sato shows that $\xi_{1,c_{\eta}}^{id,\gamma}$ have different Archimedean regulators for different values of $c_{\eta}$ and hence produce eighteen independent cycles. We expect something similar to hold in this general case as well. As a consequence, we recover the following well-known result:

\begin{cor}
	The group $H^3_{\M}(A_{\eta},\Q(2))$ is of infinite rank, where $A_{\eta}$ is the generic fibre of the universal family of products of elliptic curves over $X(2) \times X(2)$. 
\end{cor}

\begin{rem} Spiess \cite{spie} constructed motivic cycles on the product of two elliptic  curves over a local field with boundary an isogeny of the special fibre. Our cycle is closely related to his. In fact our construction should also work in the mixed characteristic case as our arguments are purely algebraic. 
	
	\end{rem}

\part{Algebraicity of values of Green's functions}

\section{Green's function and the regulator map}

\label{greensfunctions}

In the previous section, we computed the boundary map, which can be considered to be a `non-Archimedean' regulator. In this section, we compute the Archimedean regulator against a particular choice of $(1,1)$ form after restricting our cycle to the universal family over a modular curve. In \cite{sato} he computed the  regulator against a $(2,0)$ form in the case when $\alpha=1$ and is able to show it is non-trivial (and in fact depends on the choice of node $c$).

Our expression is in terms of higher Green's functions. The fact that this is the case goes back to Zhang \cite{zhan} and independently Mellit \cite{mell}. We apply our formula to a conjecture of Gross, Kohnen and Zagier which asserts that the values of higher Green's functions at CM points is the log of an algebraic number. The conjecture is now a theorem of \cite{li} and \cite{BLY} but their methods are quite different.

\subsection{Higher Green's functions}

Let $\Gamma$ be a congruence subgroup in $\SL_2(\mathbb{Z})$ and let $X=\overline{\Gamma\backslash \HH}$ denote the compactified modular curve of level $\Gamma$. Gross and Zagier \cite{grza} considered certain  `higher' Green's functions on $X \times X$ defined as follows. Let $k\geq 1$. The {\em higher Green's function of weight $k$} for $X$ is defined to be
$$
G^X_{k}(z_1,z_2)=-2 \sum_{\gamma \in \Gamma} \QQ_{k} \left( 1+\frac{|z_1-\gamma z_2|^2}{2 \Im(z_1) \Im(\gamma z_2)}\right).
$$
Here $\QQ_s$ is the Legendre function of the second kind defined by the Laplace integral 
$$
\QQ_{s-1}(t):=\int_0^{\infty} \frac{du}{(t+\sqrt{t^2-1} \cosh(u))^{s}}, \hspace {1cm} t>1,s>1.
$$

\subsection{The conjecture of Gross-Kohnen-Zagier}

In \cite{grza} and \cite{GKZ}, Gross and Zagier and later Gross, Kohnen and Zagier made precise conjectures about the values of higher Green's functions at algebraic points. The conjecture is the following.

Let $\Gamma \subset SL_2(\ZZ)$ be a congruence subgroup  and let $G_s(z_1,z_2)=G^X_s(z_1,z_2)$ be the Green's function defined above.
If $f(z)=\sum c_f(m) q^m$, where $q=e^{2\pi iz}$ is a weakly holomorphic modular form of weight $-2j$  for $\Gamma$ with $c_f(-m) \in \ZZ$ for all $m>0$ let 
$$G^f_{j+1}(z_1,z_2)=\sum_{m>0} c_f(-m)m^j G_{j+1}(z_1,T_mz_2)$$
and let 
$$|T_f|=\bigcup_{c_f(-m)\neq 0} T_m$$
where $T_m \subset X \times X$ is the $m^{th}$ Hecke correspondence. 

Then they conjecture
\begin{conj} Let $f$ be as above and assume $c_f(-m) \in \ZZ$ for all $m>0$. Let $z_1$ and $z_2$ be CM points of discriminants $D_1$ and $D_2$  respectively lying on $X\times X \backslash |T_f|$.  Then there exists an $\alpha  \in \bar{\Q}$  depending only on $D_1, D_2, j$ and $f$  such that 
	$$(D_1D_2)^{j/2} G^f_{j+1}(z_1,z_2)=\log|\alpha|$$

\end{conj}

The original conjecture was stated only for $\Gamma_0(N)$ and  in terms of a Green's function determined by `relations among coefficients of modular forms of weight $(2j+2)$' but it is known that this space can be identified with the space of weakly holomorphic forms of weight $-2j$. More precise versions of this conjecture also specify, for instance, the field containing $\alpha$. 

Several special case or conditional results have been obtained by many people - starting with the original work of  Gross-Zagier \cite{grza}, Gross-Kohnen-Zagier \cite{GKZ}, Zhang \cite{zhan}, Mellit \cite{mell}, Viazovska \cite{viaz}, Zhou \cite{zhou} and Bruinier-Ehlen-Yang \cite{BEY}. This is now a theorem of Li \cite{li} and Bruinier-Li-Yang \cite{BLY}. Their  approach is to use that the higher Green's functions can be realised as Borcherds lifts of weakly holomorphic modular forms. 

In this section we prove  special cases of this conjecture using the motivic cycles that we constructed above.  If $X=X(2)$ the modular curve with full level structure we show that the conjecture holds when $z_1$ and $z_2$ have CM by the same field. 

\subsection{CM cycles}
\label{cmcycles}

Here we largely follow Zhang \cite{zhan}. Let $X=\overline{\Gamma \backslash \HH}$ be a modular curve with sufficient level structure so that the universal family  $E \rightarrow X$ exists. Let $Y=\Gamma \backslash \HH$ be the open modular curve. Let $E_{\eta}$ be the generic fibre and $E_y$ the fibre over a point $y$ of $Y$.  Let $A_{\eta}=E_{\eta}\times E_{\eta}$. Let $\A \rightarrow X$ be the canonical compactification of $A_{\eta}$. For a point $y$ on $Y$ let $A_y=E_y \times E_y$ be the fibre over $y$. More generally, for $j\geq 0$ we consider $A_y^j,A_{\eta}^j$ and let $\A^j$ denote the canonical compactification of $A_{\eta}^j$.

The symmetric group $\MS_{2j}$ acts on $E^{2j}$. Let 
$$\epsilon_{\sgn}=\frac{1}{(2j)!} \sum_{\sigma \in \MS_{2j}} \sgn(\sigma)\sigma$$ 
be the sign idempotent in the group ring.

If $\tau$ is an imaginary quadratic number,  the elliptic curve  $E_{\tau}$ has CM by $\ZZ[\frac{-D+\sqrt{-D}}{2}]$ where $D$ is the discriminant of $\tau$. Consider the cycle 
$$Z_{\tau}=\es(\Gamma_{\sqrt{-D}}-\Gamma_{-\sqrt{-D}})$$
in $H^2_{\M}(A_{\tau},\Q(1))$ and let 
$$Z^j_{\tau}=\es \left(Z_{\tau}^{\times j}\right)$$
where the $Z_{\tau}^{\times j}$ means the $j$-fold external direct product of $Z_{\tau}$. This cycle is in $H^{2j}_{\M}(\A^j_{\tau},\Q(j))$.

The cycle $Z_{\tau}^j$ is usually called the complex multiplication (CM) cycle. It is orthogonal to the generic cycles of codimensional $j$ in $A_{\tau}^j$ and when considered as a codimensional $(j+1)$ cycle on $\A^j$ it lies in $CH^{j+1}_{\hom}(\A^j)$ - that is, it is null homologous. 

For a point $y$  on $Y$, let $\omega_y$ be the canonical differential $dz=\frac{dx}{y}$ and let $V_y=\int_{E_y} \omega_y \wedge \bar{\omega}_y$.  The group $\MS_{2j}$ acts on $H^1(E_y)^{2j}$. Let 
$$\eta^j_y= \frac{\sqrt{\binom{2j}{j}}}{V_y^j}  \es(\omega_{1,y}\wedge \bar{\omega}_{2,y}\wedge \cdots \wedge \bar{\omega}_{2j,y})$$
where $\omega_{i,y}$ means the form $\omega_y$ on the $i^{th}$ copy of $E_y$. $\eta^j_y$. This form $\eta^j_y$ is normalized so that 
$$\int_{A_y^j} \eta^j_y \wedge \eta^j_y=(-1)^j$$ 
If $\tau$ is a CM point with CM by $\sqrt{-D}$ then $V_{\tau}=|\sqrt{-D}|$ and one has 
$$cl(Z^j_{\tau})=\frac{4^j}{\sqrt{(2j)!}} |\sqrt{-D}|^j \eta^j_{\tau}$$
The upshot is that one has a $(j,j)$ form defined for all non-cuspidal $y$ such that when $y$ is a CM point a multiple of  $\eta^j_y$ is the cycle class of the CM cycle at that point. 

\begin{rem} In fact if one considers the `un-normalized' form $\tilde{\eta}^j_y=\es(\omega_{1,y}\wedge \bar{\omega}_{2,y}\wedge \cdots \wedge \bar{\omega}_{2j,y})$ one has a clean constant:
	$$cl(Z^j_{\tau})=\frac{4^j}{j!} \tilde{\eta}^j_{\tau}$$

	\end{rem}

\subsection{Hecke cycles}
\label{heckecycles}

Let $T$ be a component of a Hecke correspondence $T_n \subset X \times X$ for some $n$. If $\eta_T$ is the generic point of $T$ let $E_{\eta_T}$ denote the fibre over it and $E_T$ its canonical compactification. Let $A_{\eta_T}=E_{1,\eta_T} \times E_{2,\eta_T}$ be the fibre over $\eta_T$ and $A_T$ its compactification. There is an isogeny $\phi_{\eta_T}:E_{1,\eta_T} \rightarrow E_{2,\eta_T}$ of degree $n$ and let $\Gamma_{\phi_T}$ denote the  closure of $\Gamma_{\phi_{\eta_T}}$  in $\A_T$. 
Consider the cycle 
$$Z_T=\es(\Gamma_{\phi_T}-\Gamma_{-\phi_T}) \in H^2_{\M}(A_T,\Q(1))$$
and more generally 
$$Z_T^j=\es(Z_T^{\times j}) \in H^{2j}_{M}(A_T^j,\Q(j))$$
where $Z_T^{\times j}$ means the $j$-fold external direct product of $Z_T$. One knows, from the work of Scholl \cite{schoMot} that the projector eliminates the part of $ \Gamma_{\phi_T}-\Gamma_{-\phi_T}$ supported over the cusps. 

The cycles $Z_T^j$ are codimensional $(j+1)$ cycles in $H^{2j+2}_{\M}(\A^j,\Q(j+1))$. 
We call the divisors $T$ Hecke divisors and the cycles $Z^j_T$ Hecke cycles.

Let $g_{Z^j_T}$ denote the Green's current for  $Z^j_T$ which has the following properties 

\begin{itemize}
	
	\item $\frac{\partial \bar{\partial}}{\pi i} g_{Z^j_T} = \delta_{Z^j_T}$
	\item The integral 
	$$\int_{A_t^j(\C)} g_{Z^j_t} \eta=0$$
	for any $\frac{\partial \bar{\partial}}{\pi i}$ closed form $\eta$ on $A_t^j(\C)$ for every $t \in T$, where $Z^j_t=Z^j_T|_{A^j_t}$.
\end{itemize}
The second condition normalizes the choice of Green's current. 

For a point $(z_1,z_2)$ on $X \times X$ let $A_{z_1,z_2}=E_{z_1} \times E_{z_2}$. 
Let $\omega_{z_i}$ be the canonical holomorphic differential form on $E_{z_i}$ and let 
$V_{z_i}=\int_{E_{z_i}} \omega_{z_i} \wedge \bar{\omega}_{z_i}$. 

Let $\eta^j_{z_1,z_2}$ be the form on $A_{z_1,z_2}^j$ given by 
$$\eta^j_{z_1,z_2}=  \frac{\sqrt{\binom{2j}{j}}}{(V_{z_1}V_{z_2})^j}             \es(\omega_{1,z_1}\wedge \bar{\omega}_{2,z_2}\wedge \omega_{3,z_1}\wedge  \cdots \wedge \bar{\omega}_{2j,z_2})$$
where $\omega_{i,z_i}=\omega_{z_1}$ on a copy of  $E_{z_1}$ for odd $i$   and  $\omega_{i,z_i}=\omega_{z_2}$ on a copy of $E_{z_2}$ for even $i$.  Note that when $z_1=z_2=y$ this is  the form $V^j_y\eta^j_y$ above. The forms $\eta_{z_1,z_2}$ and $\eta_y$ share the property that they are invariant under the action of the group. 

We then have the following theorem, generalizing Zhang \cite{zhan}, Proposition 3.4.1,
	
	\begin{thm} Let $T$ be a component of $T_{n_T}$, the $n_T^{th}$ Hecke correspondence. Then we have 
		$$\int_{A^j_{z_1,z_2}} g_{Z^j_T} \eta^j_{z_1,z_2} = \frac{2^j}{\sqrt{\binom{2j}{j}}} n^j_T G_{j+1}(z_1,T z_2)$$

	\label{newgreen}
	\end{thm}

\begin{proof} 
	
	The idea behind the proof is the same as Zhang's \cite{zhan}, Proposition 3.4.1. Both sides have logarithmic singularities along $T$ and are eigenfunctions for the Laplacian and both sides are invariant under the action of $\Gamma \times \Gamma$. 
	
	If $F(z_1,z_2)=\int_{A^j_{z_1,z_2}} g_{Z^j_T} \eta^j_{z_1,z_2} $ then uniqueness of the Green's function implies that 
	$$F(z_1,z_2)=C_j G_{j+1}(z_1,Tz_2)+c_j$$ 
	for some constants $C_j$ and $c$. By restricting to the diagonal and using Zhang's Proposition 3.4.1 we get
    $$C_j=\frac{2^j}{\sqrt{\binom{2j}{j}}} n^j_T \text{ and } c_j=0.$$ 
	The factor $\frac{2^j}{\sqrt{\binom{2j}{j}}}n^j_T$ comes from the fact that Zhang uses normalized cycles. 
\end{proof}

If $\WW \rightarrow B$ is a family of varieties over a base $B$ and $\xi$ is a class in the motivic cohomology of the generic fibre, its regulator can be thought of as a function on $B$ defined at those $s \in B$ where $\xi|_s$ is defined.  
Pairing it against a form $\omega$ of the generic fibre, we get a function
$$ F(s)=\langle \reg(\xi|_s),\omega|_s\rangle.$$
This is analogous to a `normal function' in transcendental algebraic geometry - where it takes values in torus bundle - the intermediate Jacobian. In our case it takes values in the Deligne cohomolgy - which is  a generalised torus - a product of $\C^*$ and $\C$ and one can project to the `real' Deligne cohomology to get a real valued function. 

We would like to understand this function - for which we can use Theorem \ref{newgreen}. Recall that if $f$ is a function on a modular curve $X$  with divisor $\div(f)=\sum a_{\tau} {\tau}$, then, if $y$ and $y'$ are two points,  
$$\log|f(y)/f(y')|=\sum a_{\tau} (G_1(\tau,y)-G_1(\tau,y'))$$
Namely, this shows that the regulator of an element of the group $H^1_{\M}(\eta,\Q(1))=k(X)^*$ evaluated at a divisor of degree $0$ can be expressed in terms of Green's functions of degree $1$ determined by its divisor. Note that the divisor of the function only determines the function up to a constant - so it is only well defined when evaluated on a null homologous cycle. 

We have the following proposition:

\begin{prop}
	Let $X$ be a modular curve and $A^j_{\eta}$ the $j$-fold product of the  generic fibre of the universal abelian surface over $X\times X$. Let $\xi$ be a cycle in the group $H^{2j+1}_{\M}(A^j_{\eta},\Q(j+1))$ with boundary 
	$$\partial(\xi)=\sum_T a_T Z^j_T$$
	where $T$ is a component of the Hecke correspondences $T_{n_T}$  and $Z^j_T$ are the Hecke cycles defined above lying in the fibre over $T$. 
	
	Then if $\eta^j_{z_1,z_2}$ is the $(j,j)$-form defined above, 
	$$\langle \reg_{\R}(\xi),\eta^j_{z_1,z_2}\rangle=  \frac{2^j}{\sqrt{\binom{2j}{j}}} \sum a_T n_T^j G_{j+1}^X(z_1, Tz_2)$$
	as long as $(z_1,z_2)$ does not lie on $\bigcup_{a_T \neq 0} T$.
	\label{reggreen}
	
\end{prop}

\begin{proof}
	The real regulator of an  element of $H^{2j+1}_{\M}(A^j_{\eta},\Q(j+1))$ is a current on  $(j,j)$-forms in the family. From \cite{soul} Chapter III, Theorem 1, for instance,  it can also be obtained as a linear combination of the Green's currents associated to the cycles appearing in the boundary. If 
	$$\partial(\xi)=\sum a_{T} C_{T}$$
	for some irreducible divisors  $T$ on $X \times X$  and cycles $C_{T}$ in the fibres over those divisors, let $g_{C_T}$ be a Green's current associated to the cycle $C_T$. Then one has 
	$$dd^c \reg(\xi)=-\sum a_T \delta_{C_T}$$
	If, further, the boundary of $\xi$ is of the form $\sum a_T Z^j_T$ where $Z^j_T$ are  Hecke cycles, then we have the normalized Green's current $g_{Z^j_T}$ above and so $\reg(\xi)$ and $\sum a_T g_{Z^j_T}$ are both Green's currents for the same cycle. 
	
	Two Green's functions for the same cycle differ by a $dd^c$ closed form so we have 
	$$\reg(\xi)-\sum_T a_T g_{Z^j_T}=\alpha$$
	where $dd^c \alpha=0$. 
	
	From the invariant cycle theorem one has that 
	$$\int_{A_{z_1,z_2}} \alpha|_{A_{z_1,z_2}} \wedge \eta^j_{z_1,z_2}=0$$
	Hence one has 		
	$$\langle \reg(\xi),\eta^j_{z_1,z_2} \rangle = \sum_T a_T \int_{A_{z_1,z_2}}  g_{Z^j_T}\eta^j_{z_1,z_2}  $$
	From  Theorem \ref{newgreen} we get 
	$$\langle \reg(\xi),\eta^j_{z_1,z_2}\rangle = \frac{2^j}{\sqrt{\binom{2j}{j}}} \sum a_T n_T^j G^X_{j+1}(z_1,T z_2)$$
	If $\xi$ is a decomposable cycle then this expression is $0$ as the boundary does not involve any of the Hecke cycles supported on $T$. So if it is non-zero the cycle is indecomposable. In particular, this gives a well defined map on $H^{2j+1}_{\M}(A^j_{\eta},\Q(j))_{ind}$.

\end{proof}

If $\xi$ is a decomposable cycle then one can see that the regulator can be expressed in terms of $G_k$ for $k<(j+1)$. 

\section{Algebraicity Results}

\subsection{Algebraicity of values of Green's functions}

We now apply the above theorem  to show that if $\xi$ is an indecomposable cycle then the values of the higher Green's functions at certain CM points are logarithms of algebraic numbers. 
The idea here is that when $z_1$ and $z_2$ are CM points with CM by the same field,  the form $\eta^j_{z_1,z_2}$ is represented by an algebraic cycle -- so evaluating a current on the form can be translated to a question of intersection of algebraic cycles.

Let $z_1$ and $z_2$ be CM points on a modular curve $X$ with CM by the {\em same} field. Then the corresponding elliptic curves $E_{z_1}$ and $E_{z_2}$ are isogenous elliptic curves with CM and the Abelian surface $A_{z_1,z_2}=E_{z_1} \times E_{z_2}$ has Picard number $4$.

If $\phi$ is the isogeny and $E_{z_z}$ and $E_{z_2}$ have CM by $\Q(\sqrt{-D})$ then the cycle 
$$Z^j_{z_1,z_2}=\es \left( \left( \Gamma_{\phi \circ \sqrt{-D}} - \Gamma_{-\phi \circ \sqrt{-D}} \right)^{\times j}\right)$$
The cycle class of this is a multiple of the forms $\eta^j_{z_1,z_2}$. We have 
$$cl(Z^j_{z_1,z_2})= \frac{2^j}{\sqrt{\binom{2j}{j}}} D^j
\eta^j_{z_1,z_2}$$
Note that the $D^j$ is in fact made up of two $\sqrt{D}^{j}$ terms - one coming from the volume and the other from the normalization of the CM cycle. We then have the following theorem:

\begin{thm}
	Let $X$ be a modular curve and $A_{\eta}$ the generic fibre of the universal abelian surface over $X \times X$. Let $\xi=\sum(C_i,f_i)$ be a cycle in the group $H^{2j+1}_{\M}(A^j_{\eta},\Q(j+1))$ with boundary $\partial(\xi)=\sum_{T} a_{T} Z^j_{T}$, where $Z^j_{T}$ is the Hecke cycle in the fibre over a  Hecke divisor  $T$. 
	
	If $(z_1,z_2)$ is a pair of  CM points with $(z_1,z_2) \in X \times X - \bigcup_{a_T \neq 0} T$ with CM by the {\bf same field}, we have
	$$\frac{4^j}{\binom{2j}{j}} \sum a_T n_T^j D^j G_ {j+1}^X(z_1,T z_2)=  \log \prod_i  \left| f_i(C_i|_{A_{z_1,z_2}} \cap Z^j_{z_1,z_2})\right|$$
	where by $f$ evaluated at a $1$-cycle $\sum m_P P$ we mean $\prod f(P)^{m_P}$. 
	In particular, the value is the logarithm of an algebraic number. 
	\label{algebraicity}
\end{thm}

\begin{proof}

	 Since $(z_1,z_2)$ is  outside the support of the boundary of the motivic cycle, the cycle $\xi$ restricts to give a cycle $\xi|_{z_1,z_2}=\xi_{z_1,z_2}$ in the motivic cohomology group $H^{2j+1}_{\M}(A^j_{z_1,z_2},\Q(j+1))$. 
	
	There is a compatibility between the intersection product on motivic cohomology  and the cup product on Deligne cohomology, \cite{esvi}, (Proposition 7.4.). In our case this implies that  the regulator of the cycle $\xi$ evaluated at $\eta^j_{z_1,z_2}$ is then the same  as the regulator of the product of the cycles $\xi_{z_1,z_2}$ and $Z^j_{z_1,z_2}$ in the motivic cohomology groups of $A^j_{z_1,z_2}$: 
	$$H^{2j+1}_{\M}(A^j_{z_1,z_2},\Q(j+1)) \otimes H^{2j}_{\M}(A^j_{z_1,z_2},\Q(j)) \longrightarrow H^{4j+1}_{\M}(A^j_{z_1,z_2},\Q(2j+1)).$$
	Elements of the group $H^{4j+1}_{\M}(A^j_{z_1,z_2},\Q(j+1))$ are of the form $\prod_P (P,u_P)^{m_P}$ where $P$ is a point defined over possibly an extension $k_P$ of $K$, $u_P \in K_1(P)=k_P^*$ and $m_P \in \ZZ$. The regulator of such an element is
	$$\reg(\psi)=\log \left( \prod |u_P|^{m_P}\right)$$

	Explicitly, if $\xi_{z_1,z_2}=\sum (C_i|_{z_1,z_2},f_i)$ then
	$$\xi_{z_1,z_2} \cap Z^j_{z_1,z_2}=\sum (f_i, C_i|_{z_1,z_2} \cap Z^j_{z_1,z_2}).$$
	If $C_i|_{z_1,z_2} \cap Z^j_{z_1,z_2}=\sum_{P_i}  m_{P} P$ then 
	$$ \xi_{z_1,z_2} \cap Z^j_{z_1,z_2}=\prod_i \prod_{P \in Z^j_{z_1,z_2} \cap C_i}  (P,f_i(P)^{m_P})$$
	The regulator of this is 
	$$\log \left( \prod_i \prod_P   \left|f_i(P)^{n_P} \right|\right)$$
	On the other hand, from  Proposition \ref{newgreen} we have 
	$$\langle \reg_{\R}(\xi),\eta^j_{z_1,z_2} \rangle = \frac{4^j}{\binom{2j}{j}}\sum_T a_{T} n_T^j D^j G_{j+1}(z_1,Tz_2).$$
	Hence at a CM point $(z_1,z_2)$ we have two expressions for the regulator -- one in terms of a special value  of a higher Green's function and the other in terms of the regulator of an element of  $H^{4j+1}_{\M}(A^j_{z_1,z_2},\Q(2j+1))$.	
	
	Combining the two gives 
	$$\langle \reg_{\R} (\xi_{z_1,z_2}),\eta^j_{z_1,z_2} \rangle=\frac{4^j}{\binom{2j}{j}} \sum a_{\tau} n_T^j D^j G_{j+1}(z_1,Tz_2))=  \log  \prod_i \prod_P \left| f_i(P)^{m_P}\right|$$
	At a CM point $(z_1,z_2)$, the element $\xi_{z_1,z_2}$  and the cycle $Z^j_{z_1,z_2}$ are defined over a number fields. Hence so is $C_i \cap Z^j_{z_1,z_2}$ and the numbers $u_P$ are algebraic numbers.  Therefore  the number 
	$$ \alpha = \prod_i \prod_P  f_i(P)^{m_P}$$
	is the value of an algebraic function evaluated at algebraic points and is hence an algebraic number.

\end{proof}

\subsection{The conjecture of Gross-Kohnen-Zagier.}

The theorem above links cycles in the group $H^{2j+1}_{\M}(A^j_{\eta},\Q(j+1))$ with algebraicity results for Green's functions. 

If $\xi$ is a decomposable cycle then the regulator computed against $\eta_{z_1,z_2}^j$ is $0$ so one cannot get any algebraicity results. Hence we need indecomposable cycles. 
In Section \ref{motiviccycleconstruction} we constructed such cycles in the case when $j=1$ and $X=X(2)$ and we can use them to prove algebraicity results for $G_2$. This results in the following theorem:

\begin{thm} Let $z_1,z_2$ be two CM points with CM by the same field $K=\Q(\sqrt{-D})$. Then  if $T_n^{\gamma}$ is a component of the Hecke correspondence $T_n$ on $X \times X$, where $X=X(2)$ is the Abelian surface with full level 2 structure, one has 
	$$2Dn G_2(z_1,T_n^{\gamma}z_2)=\log|\alpha|$$
	for some algebraic number $\alpha$. 

\label{zagierconjecture}
\end{thm}
\begin{proof}

We constructed the cycle $\xi_{n,c_{\eta}}^{\phi,\gamma}$ in the group $H^3_{\M}(\TK_{A_{\eta}},\Q(2))$, the generic  Kummer $K3$ surface of a product of elliptic curves. By a standard argument we can pull this cycle back to obtain a cycle $\xi$ in $H^3_{\M}(A_{\eta},\Q(2)$, where $A_{\eta}=E_{1,\eta} \times E_{2,\eta}$ is the generic product of elliptic curves. 

From Theorem \ref{mainthm}, we have that the boundary is supported on the fibres over $T^{\gamma}_n$ and 
$$\partial(\xi)=(\Gamma_{\phi_{T_n^{\gamma}}}-\Gamma_{-\phi_{T_n^{\gamma}}})$$
Applying the projector $\es$  to this cycle, as it commutes with the map $\partial$  we get  %
$$\partial(\es(\xi))=\es(\Gamma_{\phi_{T_n^{\gamma}}}-\Gamma_{-\phi_{T_n^{\gamma}}})=Z_{T^{\gamma}_n}$$ 
So we are precisely in the situation of Proposition \ref{reggreen}. Hence the regulator of this cycle can be expressed in terms of Green's functions of weight $2$,
$$\langle \reg(\es(\xi)),\eta^1_{z_1,z_2} \rangle= \sqrt{2} n G_2(z_1,T^{\gamma}_n z_2)$$
Let $(z_1,z_2)$ be a pair of CM points of the same discriminant  not lying on $T_n^{\gamma}$.  Since $cl(Z_{z_1,z_2})=\sqrt{2} D  \eta_{z_1,z_2}^1$, applying  Theorem \ref{algebraicity} we get 
$$2Dn G_2(z_1,T_n^{\gamma}z_2)=\log|\alpha|$$

\end{proof}

Since there are no cusp forms of weight $4$ the conjecture of Gross-Kohnen and Zagier implies that the Hecke correspondences $T_n$ act by $0$ on $S_4(\Gamma(2))$. This translates to the fact that there is a weakly holomorphic modular form $f_n$ of weight $-2$ with principal part $\frac{1}{q^n}$ and one should have 
$$2nDG^{f_n}_2(z_1,z_2)=2nDG_2(z_1,T_nz_2)=\log|\alpha|$$
Summing up the terms over all the  components  $T^{\gamma}_n$ of $T_n$ we get 

\begin{thm} If $(z_1,z_2)$ are CM points of the same discriminant $D$ then 
	$$2DnG_2(z_1,T_nz_2)=\log|\alpha|$$
	for an algebraic number $\alpha$. 
	
\end{thm}

This was  proved by Viazovska \cite{viaz} in her thesis using Borcherd's lifts and not only in for $G_2$ but for all $G_k$. Mellit had proved special cases of $G_2$ along the same lines as our results - namely by constructing motivic cycles. Mellit's case is essentially the case obtained by restricting our cycles to the diagonal, (or more generally a modular curve). In the diagonal case, for instance, we have 
$$V_z G_2(z,Tz)=\sum_{t \in T \cap diag} \sqrt{D_t}G_2(t,z)$$
where $t$ runs through all the (CM) points of intersection of $T$ and the diagonal. Of course the work of Li and Bruinier-Li-Yang supersedes all these cases. 

However, from our motivic methods one gets a little more as we obtain not only the logarithm of $|\alpha|$ but the number itself. 

The automorphic methods of Li and Bruinier-Yang-Li do not distinguish between the cases of when $(z_1,z_2)$ have the same discriminant or not - though Viazovska's work made that distinction. In our case we need to intersect with the CM cycle and that does not exist in the case when $z_1$ and $z_2$ have different discriminants.

\subsection{A generalization}

In fact, our construction and argument above proves a little more. Instead of restricting to CM points we can restrict to Hecke cycles distinct from the Hecke cycles appearing in the boundary. Since the class of a Hecke cycle $Z^j_{T}|_{z_1,z_2}$ is represented by $\eta^j_{z_1,z_2}$ the same argument as above gives us the following theorem: 

\begin{thm} Let $\xi$ be a motivic cycle in $H^{2j+1}_{\M}(A_{\eta},\Q(j+1))$ with boundary 
	$$\partial (\xi)=\sum a_T Z^j_T$$
	Let $Z^j_{T'}$ be a Hecke cycle distinct from the $Z^j_T$ appearing in the boundary. Then 
	$$\langle \reg(\xi),\eta^j_{z_1,z_2}\rangle = \frac{2^j}{\binom{2j}{j}} \sum_T a_T n^j_T G_{j+1}(z_1,Tz_2) = (V_{z_1}V_{z_2})^{-j}\log|f(z_1,z_2)|$$
	for $f$ a rational function in $\bar{\Q}(T')^*$ and $(z_1,z_2)$ on $T'$.  
	\end{thm}
	Note that if $z_1$ and $z_2$ are CM points on $T'$ then they necessarily have the same field of CM.

\subsection{Arakelov pairing}

The above result can best be explained using Arakelov theory, which leads to a second expression for the value of the Green's function. It is as follows.  

One has the Arakelov pairing of two arithmetic cycles $\TC_1$ and $\TC_2$  in the arithmetic variety $\XX/_{\ZZ}$ defined as follows. 
$$\langle C_1,C_2\rangle_{Ar}=\langle C_1,C_2\rangle_{\infty} + \sum_p \langle C_1,C_2\rangle_p$$
where
$$\langle C_1,C_2 \rangle_{\infty} = \langle g_{C_1},\omega_{C_2}\rangle = \langle \omega_{C_1},g_{C_2}\rangle$$ 
and 
$$\langle C_1,C_2 \rangle_p = \langle C_{1,p}.C_{2,p}\rangle \log(p)$$
where $\langle . \rangle$ is the intersection pairing mod $p$. 

If $C_1=\partial(\xi)$ is the boundary of a motivic cycle then it is rationally equivalent to $0$, so $(C_1,C_2)_{Ar}=0$ and 
$$\langle C_1,C_2 \rangle_{\infty}=-\sum_p \langle C_1,C_2 \rangle_p$$
As noted above the regulator is a Green's current for the boundary of the motivic cycle so 
$$\langle C_1,C_2 \rangle_{\infty} = \langle \reg(\xi),\omega_{C_2}\rangle =\langle \partial(\xi),g_{C_1}\rangle $$
Let is apply this in the case when $C_1=\sum a_T Z_T|_{A_{z_1,z_2}}$ and $C_2=Z_{z_1,z_2}$. Then we get 
$$\langle \sum a_T Z_T|_{z_1,z_2},Z_{z_1,z_2} \rangle_{\infty} = \int_{A_{z_1,z_2}} \reg(\xi) \omega_{z_1,z_2}=\sum a_T  \int_{A_{z_1,z_2}} Z_T|_{z_1,z_2}g_{Z_{z_1,z_2}}$$
$$=\frac{4^jD^j}{\binom{2j}{j}} \sum a_T n^j G_{j+1}(z_1,Tz_2)$$
This gives us 
$$               \frac{4^jD^j}{\binom{2j}{j}} \sum a_T n^j G_{j+1}(z_1,Tz_2)= \langle \partial(\xi),Z_{z_1,z_2} \rangle_{\infty}=-\sum_p \langle \partial(\xi),Z_{z_1.z_2} \rangle_p \log(p)$$
Our expression above is in terms of the logarithms of the values of certain functions at some points of intersection of algebraic varieties. To reconcile the two we observe the following. 

If $Z=Z_{z_1,z_2}$ is the CM cycle at $(z_1,z_2)$ then one has 
$$dd^c g_Z+ \delta_Z=\omega_Z$$ 
so this gives
$$\langle \partial(\xi),Z \rangle_{\infty}= \langle \reg (\xi),\omega_Z \rangle $$
$$=\langle \reg(\xi),dd^c g_Z \rangle + \langle \reg(\xi),\delta_Z \rangle$$
Now $ \langle \reg(\xi),dd^c g_Z \rangle = \langle dd^c \reg(\xi),g_Z \rangle = 0$ as the boundary of $\xi$ is rationally equivalent to $0$. Hence we are left with 
$$\langle \partial(\xi),Z \rangle_{\infty}= \langle \reg(\xi),\delta_Z \rangle $$
which the expression we had above. The equality of these two seemingly different expressions - one in terms of the intersection number mod $p$ and the other in terms of the value of functions at points of intersection  can be seen by the product formula. 

The upshot of this is that the algebraicity of the value of Green's functions is directly a consequence of the existence of motivic cycles. 

The advantage of the first point of view is that it suggests that the argument can be extended to the case when $(z_1,z_2)$ have different discriminants by viewing $\langle \reg(\xi),\eta^j_{z_1,z_2}\rangle$ as the Archimedian part of an Arakelov pairing. Since for primes of supersingular reduction there is an honest cycle the pairing 
$$\langle \sum a_T (Z_T)_p,(Z_{z_1,z_2})_p \rangle_p$$
and so the same argument would imply algebraicity in this case as well. We hope to return to this case in the near future.

\
\section{Reconciling the two approaches}

\subsection{Hecke cycles and Hecke operators}

Zagier's conjecture is formulated in terms of relations among Hecke operators. In other words, if $\{a_n\}$ is a relation among Hecke operators acting on modular forms of weight $k$ of the form 
$$\sum a_n T_n=0$$
then one expects the Green's function $\sum a_n n^{k-1}G_k(z_1,T_n z_2)$ to have algebraic values when $(z_1,z_2)$ are CM points. 

A natural question to ask is - on weight $2$ forms the Hecke action is induced by the Hecke correspondences $T_n \subset X \times X$ on the cohomology group $H^1(X,\Q)$. What  are the  geometric objects underlying the Hecke action on weight $k$ forms?

Weight $k$ forms can be realized in the cohomology of $\E^{(k-2)}$ where as before $\E^j$ is the canonical compactification of the universal family of self products of elliptic curves over $X$. The Hecke operator can be realized as the correspondence coming from the cycle $C_m \subset (\E_1 \times \E_2)^{(k-2)} \rightarrow X \times X$ given by $\Gamma_{\phi^{k-2}}$ where $\phi$ is the isogeny corresponding to $T_m$.  Here $\E_1$ denotes $\E$ over the first copy of $X$ and $\E_2$ denotes $\E$ over the second. Fibrewise, over a point $(z_1,z_2)$ lying on $T_m$ this is 
$$((x_1,\phi(x_1)),\dots, (x_{k-2},\phi(x_{k-2}))) \subset (E_{z_1}\times E_{z_2})^{k-2}$$
where $x_i \in E_{z_1}$. This is a codimensional $(k-1)$ subvariety of a $(2k-2)$ dimensional variety. 

Cusp forms of weight $k$ can be identified with differential forms of type $(k-1,0)$ on $\E^{k-2} \rightarrow X$ given by 
$$ f \rightarrow \omega_f=f(z)dzdt_1\cdots dt_{k-2}$$
One has the two projection maps 
$$\pi_i:((\E_1 \times \E_2)^{(k-2)} \rightarrow X \times X) \longrightarrow  \E^{(k-2)} \rightarrow X$$
$$\pi_1((z_1,z_2),(x_1,y_1),\dots (x_{k-2},y_{k-2}))=(z_1,x_1,\dots,x_{k-2})$$
and similarly for $\pi_2$. 

The cycle $C_m$ induces an algebraic correspondence $H^*(\E_1^{k-2},\C) \rightarrow H^*(\E_2^{k-2},\C)$. The Hecke action of $T_m$  is the restriction of this to  $H^{k-1,0}(\E^{k-2})$. That is 
$$ \omega_{T_m(f)} = \pi_{2,*}( \pi_1^*(\omega_f) \cup cl(C_m))$$ 
However, our Hecke cycle is not on $\E^{k-2} \times \E^{k-2}\rightarrow X \times X $ but on $\E^{k-2} \rightarrow X \times X$. To relate the two, we observe the following. Let $\Delta$ be the diagonal embedding  
$$(\E^{(k-2)} \rightarrow X 
\times X)  \hookrightarrow ((\E_1 \times \E_2)^{(k-2)} \rightarrow X \times X)$$
given by 
$$\Delta((z_1,z_2),x_1,\dots,x_{k-2})=((z_1,z_2),(x_1,x_1),\dots,(x_{k-2},x_{k-2}))$$
Then we have $\Delta^*(C_m)=\Gamma_{\phi}^{\times (k-2)/2}$ and 
$$Z_{T_m}=\es(\Delta^*(C_m))$$
If $g$ is a weakly holomorphic modular form of weight $2-k$ and $\sum c(-m)m^{k-1}T_m$ is the corresponding relation among the Hecke correspondences induced by $g$ then one has 
$$\sum c(-m) C_m=0$$
at least as correspondences acting on the space of modular forms. One then has 
$$\es(\Delta^*(\sum c(-m)C_m))=\sum c(-m)Z_{T_m}=0$$
as a correspondence. If one had 
$$\sum c(-m)Z_m\sim_{rat} 0$$
this would mean there is a motivic cycle $\xi_f$ such that $\partial(\xi_f)=\sum_m c(-m)Z_{T_m}$ and 
$$\reg(\xi_g,\eta_{z_1,z_2})=\Phi(g)(z_1,z_2)$$
where $\Phi$ is the Borcherds lift considered by Bruinier-Li-Yang \cite{BLY}. 

In the other direction, if $\xi$ is a generically indecomposable motivic cycle, then $\es(\xi)$ satisfies 
$$\partial(\es(\xi))=\sum a_T Z_{T}$$
with not all $a_T=0$, where $Z_{T}$ are the components of Hecke cycles. Averaging over  the action of $\Gamma(1)/\Gamma(2)$ this gives a cycle such that  $\xi_{sp}$ such that 
$$\partial(\xi_{sp})=\sum a_m Z_{T_m}$$ 
where $Z_{T_m}$ are the Hecke cycles over Hecke correspondences, not merely components of Hecke correspondences. We call such cycles {\em special} motivic cycles. 

If the diagonal pull-back $\Delta^*$ is {\em injective}, this would imply 
$$\sum a_m C_m \sim_{rat} 0 $$
and $\sum a_m C_m=0$ as a correspondence on modular forms. This implies that $\sum a_m q^{-m}$ is the principal part of a weakly holomorphic modular form $g$. 

This suggests that there is a dictionary 
$$ g \longleftrightarrow \xi_g $$
between weakly holomorphic modular forms and special motivic cycles. Let us call a Green's function {\em motivic} if it can be realized as the regulator of a motivic cycle and {\em automorphic} if it can be realized as the Borcherds lift of a weakly holomorphic modular form. We can then conjecture 

\begin{conj}
	Let $g$ be a weakly holomorphic modular form of weight $-2j$ for $\Gamma$. Then there is a motivic cycle $\xi_g \in H^{2j+1}_{\M}(A^j_{\eta},\Q(j))$ such that the motivic Green's function associated to $\xi_g$ and the automorphic Green's function associated to $g$ coincide:
$$\langle\reg(\xi_g),\eta^j_{z_1,z_2}\rangle=\Phi(g,(z_1,z_2))=G^g_{j+1}(z_1,z_2)$$
\label{formscycleconjecture}
\end{conj}
This suggests a motivic interpretation of weakly holomorphic modular forms. This can also be understood in terms of a relation between the localization sequence and the short exact sequence relating weakly holomorphic forms of weight $k$, weak Maass forms of weight $k$  and cusp forms of weight $(2-k)$.

\subsection{The modular  complex.}

Ramakrishnan \cite{rama} introduced the notion of a `modular' complex - a subcomplex of the complex that defines motivic cohomology (or higher Chow groups) of a Shimura variety. This is the subcomplex defined by Shimura subvarieties, boundary components, Hecke translates of all of these and the relations between them. He constructs elements in this modular complex. In fact all known examples of non-trivial elements of the motivic cohomology of Shimura varieties lie in this complex. 

Analogously, in Kuga fibre varieties over a  Shimura variety, one can consider the subcomplex generated by modular subvarieties and special cycles over them, like the CM cycles. The conjecture suggests that there should be a corresponding sequence of modular forms and that the regulators of these cycles can be computed in terms of Borcherds lifts of modular forms. In particular, the regulator of the cycles constructed by \cite{rama} should be given by Borcherds lifts. Recently \cite{ma2026} has made some progress on such results.

\section{Generalizations}

In this paper we considered the case of elements in the group $H^3_{\M}(E_1 \times E_2,\Q(2))$. There are several directions one can generalize this. Our construction is done in the Kummer K3 surface of the product of elliptic curves. One way is to consider more general $K3$ surfaces.

\subsection{More general $K3$ surfaces}

In \cite{sree-simple} we made an analogous construction of motivic cycles on  Kummer surfaces of simple Abelian surfaces. We expect they   imply similar algebraicity results for Bruinier's automorphic  Green's functions on Hilbert modular surfaces as our cycles have boundary on Hecke cycles over Hirzebruch-Zagier modular curves on Hilbert modular surfaces.

In \cite{sreeK3} we extended the construction to degree-$2$ $K3$ surfaces, that is,  double covers of $\CP^2$ branched over a sextic; the Kummer case arises when the sextic degenerates to six lines.

Different singular sextics give rise to $K3$ surfaces which are double covers of del Pezzo surfaces. The moduli of such surfaces correspond to orthogonal Shimura varieties of type $(2,n)$ for $n$ up to $19$ \cite{yuzh}. In \cite{sree-delpezzo} we use $(-1)$-curves to construct several cycles in the generic $K3$ double cover of a del Pezzo surface. Here too, if one knows the existence of rational curves of certain types - as one might expect from enumerative geometry - one can generalize the construction to get infinitely many independent cycles. Since the moduli are orthogonal Shimura varieties as well, perhaps  Borcherds lifts of weakly holomorphic forms are related to regulators of these cycles too. 

A slightly different special case is to consider $K3$ surfaces obtained as the double cover of $\CPP$ ramified at a curve of bidegree $(4,4)$. This is a generalization of the Kummer surface of a product of elliptic curves studied here. The same construction would work there to determine interesting motivic cycles.

\subsection{Higher dimensions}

The other direction is for higher dimensions. In all these cases we have considered the case of $j=1$. However, Conjecture \ref{formscycleconjecture} suggests that there should be motivic cycles in $H^{2j+1}_{\M}(A^j_{\eta},\Q(j))$ for all $j$.

In work in progress \cite{sree-products} we construct cycles in the group $H^5_{\M}(E_{\eta}^4,\Q(3)$ and more generally in the group $H^5_{\M}(A^2_{\eta},\Q(3))$ where $A$ is an Abelian surface - which should provide algebraicity results for $G_3$. Curiously, in this case the space of cusp forms of weight $6$ for $\Gamma(2)$  is $1$ dimensional, so one should not expect $n^2G_3(z_1,T_nz_2)$ to be the log of an algebraic number - but one should expect 
$$a n^2G_3(z_1,T_nz_2)-bm^2G_3(z_1,T_mz_2)=\log|\alpha|$$
for suitable integers $a$ and $b$. That is reflected in our construction.

\part{Appendix - Kannappan Sampath}

Let $n_1, n_2$ and $n_3$ be positive integers; suppose also that $n_1$ is odd. These determine Hecke correspondences $X_{n_i}$ in the surface $X(2) \times X(2)$ corresponding to those pairs of elliptic curves which have an isogeny of degree $n_i$ between them.

Recall from the previous sections, that for a pair of elliptic curves $E_1$ and $E_2$ over $\mathbf{Q}$ with an isogeny $\phi: E_1 \to E_2$, we have the following diagram:
\begin{center} 
\begin{tikzcd}
  \Gamma_\phi \ar[hook]{r} \ar{d} & E_1 \times E_2 \ar{d}{\pi}\\
  C_\phi \ar[hook]{r} \ar{d} & K_{E_1 \times E_2} \ar{d}{\psi} & \widetilde{K}_{E_1 \times E_2} \ar{l}{\text{blow-up}} \\
  Q_\phi \ar[hook]{r} & K_{E_1} \times K_{E_2} \arrow[r, symbol=\simeq] & \mathbb{P}^1 \times \mathbb{P}^1
\end{tikzcd} 
\end{center}
Here, as before,
\begin{itemize} 
\item $\Gamma_\phi$ denotes the graph of the isogeny $\phi$,
\item $K_{E_1 \times E_2}$ denotes the Kummer surface associated to the abelian surface $E_1 \times E_2$,
\item $K_{E_i}$ is the Kummer line associated to $E_i$ (abstractly isomorphic to $\mathbb{P}^1$),
\item the maps $\pi$ and $\psi \circ \pi$ are the canonical quotient maps and their geometry is explicitly described at the beginning of Section~\ref{sec-mot-cycles-on-Km-surfaces},
\item $C_\phi$ and $Q_\phi$ are the images of $\Gamma_\phi$ under $\pi$ and $\psi \circ \pi$ respectively; as discussed in Section~\ref{sec-mot-cycles-on-Km-surfaces}, the curve $C_\phi$ is a rational curve in $K_{E_1 \times E_2}$ and $Q_\phi$ is a rational curve of bidegree $(\deg \phi, 1)$ in $\mathbb{P}^1 \times \mathbb{P}^1$,
\item $\widetilde{K}_{E_1 \times E_2}$ is the Kummer K3 surface associated to $E_1 \times E_2$ obtained by blowing up the 16 points $\{\pi(\alpha_1, \alpha_2): \alpha_i \in E_i [2]\}$.  
\end{itemize}
Let $\widetilde{C}_{\phi}$ denote the strict transform of $C_\phi$ in the Kummer K3 surface $\widetilde{K}_{E_1 \times E_2}$.  As in Lemma~\ref{lem-higher-chow-cycles-on-KmK3}, $n_1$ determines a motivic cycle $\xi_{n_1} \in H^3_{\mathcal{M}}(\widetilde{K}_{E_1 \times E_2}, \mathbb{Q}(2))$ defined for those pairs $(E_1, E_2)$ of elliptic curves whose moduli point is in $(X(2) \times X(2)) \setminus X_{n_1}$. 

For pairs $(E, E')$ of elliptic curves in $X_{n_i}$ (with $i = 2, 3$), we have cycles $\Gamma_{\phi_i} - \Gamma_{-\phi_i}$ in $E \times E'$ coming from the isogeny $\phi_i: E \to E'$ (of degree $n_i$) and the corresponding cycles $\widetilde{Z}_{\phi_i} = \widetilde{C}_{\phi_i}-\widetilde{C}_{-\phi_i}$ in the Kummer K3 surface $\widetilde{K}_{E \times E'}$.

The cycles $\widetilde{Z}_{\phi_i}$ are orthogonal to the generic Neron-Severi group as well as to each other on the supersingular Kummer K3s associated to those points in $X_{n_2} \cap X_{n_3}$. This implies that, on $X_{n_3}$, the cycle $\widetilde{Z}_{\phi_2}$ is a multiple of the CM cycle at a point on $X_{n_2} \cap X_{n_3}$.

Recall that, for a surface $S$, if $\xi = \sum_i (C_i, f_i) \in H^3_{\mathcal{M}}(S, \mathbb{Q}(2)) \simeq \CH^2(S, 1)_{\mathbb{Q}}$ and $Z \in \CH^1(S) = \CH^1(S, 0)$ is a codimension $1$ cycle in $S$, then their intersection is given by
\begin{align}
  \xi \cap Z = \sum_{i} \prod_{x \in C_i \cap Z} f_i(x).
\end{align}
where, for a $0$-cycle $x = \sum a_j x_j$, the notation $f_i(x)$ stands for $\prod_j f_i(x_j)^{a_j}$. In particular, the values of $\xi_{n_1} \cap \widetilde{Z}_{\phi_2}$ defines a function on $X_{n_2}$ defined on $X_{n_2} \cap X_{n_1}^c$. The values of this function at a point $\tau \in X_{n_2} \cap X_{n_3} \cap X_{n_1}^c$ is an algebraic number; moreover, it should be an integral multiple of the value of the higher Green's function of degree $2$ corresponding to the higher Chow cycle $\xi_{n_3}$ at the CM cycle in the Kummer K3 surface corresponding to $\tau$.

Our goal here is to determine the function $\xi_{n_1} \cap \widetilde{Z}_{\phi_2}$ in the case $n_1 = 1$ and $n_2 = 3$. Note now that $\xi_{n_1}$ is Sato's cycle and is discussed as a special case of the constructions of this paper in Section~\ref{ssec-enumerative-geom-ratl-curves-p1-p1}. It will be denoted $\xi_{I}$ in the sequel (where $I$ stands for the identity isogeny).

\subsection{The case $n_1 = 1$ and $n_2 = 3$} Recall from Lemma~\ref{lem-higher-chow-cycles-on-KmK3} that the higher Chow cycle $\xi_I \in H^3_{\mathcal{M}}(\widetilde{K}_{E_1 \times E_2}, \mathbb{Q}(2))$ is defined for all pairs $(E_1, E_2)$ away from the diagonal (also denoted $X_{n_1}$ above) in $X(2) \times X(2)$; moreover, as in Proposition~\ref{prop-higher-chow-cycles-blow-up}, the higher Chow cycle $\xi_{I}$ is determined by a choice of a node $P$ in the rational curve $C_I$ and is explicitly the following
\[\xi_I = (\widetilde{C}_I, f) + (E_P, g)\]
where
\begin{itemize}
\item $\widetilde{C}_I$ is the strict transform of the locus $C_I$ (and is in fact the ``(embedded) normalization'' of the nodal rational curve $C_I$), 
\item $E_P$ is the exceptional fiber over the point $P$ in $\widetilde{K}_{E_1 \times E_2}$ (meeting $\widetilde{C}_I$ at the two points in $\widetilde{C}_I$ lying over $P$),
\item the function $f$ defined on $\widetilde{C}_I$ is such that $\divsr(f)= P_1 - P_2$ where $P_1, P_2$ are the two points in $\widetilde{C}_I \cap E_P$ mapping to $P$ in $C_I$,
\item the function $g$ defined on $E_P$ is such that $\divsr(g) = P_2 - P_1$.
\end{itemize}
For the purpose of describing the function $f$, it is convenient to work the curves $E_1$ and $E_2$ in Legendre form as the Legendre family is naturally identified with the moduli space $X(2)$. 

\subsection{Some computations with elliptic curves in Legendre form} Let $E_\lambda$ denote the elliptic curve in $\mathbb{P}^2_{[X:Y:Z]}$ given by
\begin{equation}
Y^2Z = X(X-Z)(X-\lambda Z) 
\end{equation} 
with $\lambda \neq 0, 1$.

Let $\lambda_1$ and $\lambda_2$ be the Legendre parameters of the elliptic curves $E_1 \subset \mathbb{P}^2_{[X_1:Y_1:Z_1]}$ and $E_2 \subset \mathbb{P}^2_{[X_2:Y_2:Z_2]}$ respectively. Let
\[F_i(x) = x (x - 1)(x- \lambda_i)\]
so that on the affine patch $Z_i = 1$, the curve $E_i$ is given by $y^2 = F_i(x)$. The Kummer surface $K_{E_1 \times E_2}$ is then a double cover of $K_{E_1} \times K_{E_2}$ branched at the configuration $\#_{\lambda_1, \lambda_2}$ of eight lines discussed in Section~\ref{sssec-km-surface-prod-ell-curves}; canonically it is a divisor in the total space $\mathbf{T} = \mathbb{V}(\mathcal{O}_{\mathbb{P}^1 \times \mathbb{P}^1}(2, 2))$ of the line bundle $\mathcal{O}_{\mathbb{P}^1 \times \mathbb{P}^1}(2, 2)$ over $\mathbb{P}^1_{[X_1:Z_1]} \times \mathbb{P}^1_{[X_2:Z_2]}$; it is useful to view $\mathbf{T}$ as being a quasi-projective subvariety of the weighted projective space $\mathbb{P}(1,1,1,1,4)$ and endowed with the coordinates $(X_1X_2, X_1Z_2, X_2Z_1, X_2 Z_2, w)$. In fact, $\mathbf{T}$ is cut out by the equation 
\[w^2 = Z_1Z_2\Phi_1(X_1, Z_1) \Phi_2(X_2, Z_2)\]
 where $\Phi_i(X_i, Z_i) = Z_i^3F_i(X_i/Z_i)$. The curve $C_I$ is the locus in $K_{E_1 \times E_2}$ covering the diagonal in $K_{E_1} \times K_{E_2}$ so it is given by
\[w^2 = Z_1^2F_1(X_1, Z_1) F_2(X_1, Z_1).\]
The nodes in $C_I$ lie over the points $(0, 0), (1, 1)$ and $(\infty, \infty)$ in $\mathbb{P}^1 \times \mathbb{P}^1$.

We will now describe the strict transform $\widetilde{C}_I$ of $C_I$ in the blowup of $K_{E_1 \times E_2}$ at the nodes in $C_I$. 
Over the open subset $U_{Z_1Z_2} = \{Z_1 \neq 0, Z_2 \neq 0\}$ of $K_{E_1 \times E_2}$,  the variety $\widetilde{C}_I$  can be thought of as the blow-up of the plane $\frac{X_1}{Z_1} = \frac{X_2}{Z_2}$ in the node of $C_I$ over $P = (0, 0)$. Accordingly, we will work with local coordinates $(x, w)$ for the plane $\frac{X_1}{Z_1} = \frac{X_2}{Z_2}$ where $x = \frac{X_1}{Z_1} = \frac{X_2}{Z_2}$. In sum, we see that $\left.\widetilde{C}_I\right|_{U_{Z_1Z_2}}$ is a subscheme of $\mathbb{A}^2_{(x, w)} \times \mathbb{P}^1_{[u : v]}$ cut out by 
\[uw = v F_1(x), \quad w^2 = F_1(x) F_2(x).\]
The exceptional fiber $E_P$ meets $\widetilde{C}_I$ where $u = 1$, which gives us $v^2 = \frac{F_2(x)}{F_1(x)}$ over the node and $v = \frac{w}{F_1(x)}$ over other points. The points $P_1$ and $P_2$ lying in $\widetilde{C}_I \cap E_P$ over the node $P = (0, 0)$ are $((0, 0), \pm \sqrt{\frac{\lambda_2}{\lambda_1}})$ and the function $f$ may be taken to be
\[f = \frac{\sqrt{\lambda_1} v + \sqrt{\lambda_2} }{\sqrt{\lambda_1} v - \sqrt{\lambda_2}}.\]

Similarly, we can describe $\widetilde{C}_I$ over the open subset $U_{X_1X_2} = \{X_1 \neq 0, X_2 \neq 0\}$ of $K_{E_1 \times E_2}$ containing the node over $(\infty, \infty)$; $\left.\widetilde{C}_I\right|_{U_{X_1X_2}}$ is a subscheme of $\mathbb{A}^2_{(z, w)} \times \mathbb{P}^1_{[u' : v']}$, with $z = \frac{Z_1}{X_1} = \frac{Z_2}{X_2}$, and is cut out by 
\[u'w = v' (1 - \lambda_1z), \quad w^2 = (1  - \lambda_1z)(1 - \lambda_2z).\]

We will note now that $\widetilde{Z}_\phi$ does not meet $E_P$. Note that the curve $\widetilde{C}_{\pm \phi}$ maps to $Q_\phi = (\psi \circ \pi)(\Gamma_\phi)$ while the exceptional fiber lies over $P = (0, 0)$. We will argue that $\phi$ may be chosen so that $Q_\phi$ does not contain $P$---indeed, if $e$ is a 2-torsion point, then so is $\phi(e)$ (since $\deg \phi$ is odd), and
if $e_0^i, e_1^i$, and $e_{\lambda_i}^i$ be the three non-trivial torsion points on $E_i$, we see that we may take $\phi$ such that $\phi(e_0^1) \neq e_0^2$. Since $\psi \circ \pi$ is the $x$-coordinate map, it follows that $P = (0, 0)$ does not lie on $Q_\phi$.

We now calculate:
\begin{align}
 \xi_I \cap \widetilde{Z}_\phi &= \prod_{z \in \widetilde{C}_I \cap \widetilde{Z}_\phi} f(z) \notag \\
  &=  \prod_{z \in \widetilde{C}_I \cap \widetilde{C}_\phi} f(z)  \prod_{z \in \widetilde{C}_I \cap \widetilde{C}_{-\phi}} f(z)^{-1} 
\end{align}
We compute the intersection $\widetilde{C}_I \cap \widetilde{C}_{\pm \phi}$ by first computing the intersection $Q_I \cap Q_\phi$. Since $Q_I$ is the diagonal in $\mathbb{P}^1\times \mathbb{P}^1$, and $Q_\phi = \{(x(\alpha), x(\phi(\alpha))): \alpha \in E_1\} \cup \{(\infty, \infty)\}$, the intersection $Q_I \cap Q_\phi$ consists of the point $(\infty, \infty)$ and the remaining points are parametrized by those points $\alpha \in E_1$ such that $x(\alpha) = x(\phi(\alpha))$.

The fiber in $\widetilde{K}_{E_1 \times E_2}$ of $\widetilde{C}_I \cap (\widetilde{C}_\phi \cup \widetilde{C}_{-\phi})$ over a point $x = x(\alpha)$ in $Q_I \cap Q_\phi$ consists of two points; one of these point lies on $\widetilde{C}_\phi$ and the other on $\widetilde{C}_{-\phi}$; call them $P_\phi^x$ and $P_{-\phi}^{x}$ respectively. They are explicitly (in the affine chart $\{u = 1\}$) given by:
\begin{align*}
  P_{\pm \phi}^x = \left((x, \sqrt{F_1(x)}), (x, \pm\sqrt{F_2(x)}), \pm\sqrt{\frac{F_2(x)}{F_1(x)}}\right). 
\end{align*}
Similarly, the fiber in $\widetilde{K}_{E_1 \times E_2}$ of $\widetilde{C}_I \cap (\widetilde{C}_\phi \cup \widetilde{C}_{-\phi})$ over $(\infty, \infty)$ consists of two points; viz.,  
\begin{align*}
  \widetilde{P}_{\pm \phi} = (0, 0, \pm 1). 
\end{align*}

Now, we have that
\[f(P_{\pm \phi}^x) = \frac{\pm \sqrt{\lambda_1 F_2(x)} + \sqrt{\lambda_2 F_1(x)} }{\pm \sqrt{\lambda_1 F_2(x)} - \sqrt{\lambda_2 F_1(x)} } \qquad \text{ and }  \qquad f(\widetilde{P}_{\pm \phi}) =  \frac{\pm \sqrt{\lambda_1} + \sqrt{\lambda_2} }{\pm \sqrt{\lambda_1} - \sqrt{\lambda_2} }\]
and the function $\xi_I \cap \widetilde{Z}_\phi$ is
\begin{align}
  \xi_I \cap \widetilde{Z}_\phi &=  \prod_{x \in Q_I \cap Q_\phi} \frac{f(P^x_{\phi})}{f(P^x_{-\phi})} \\
  &= \left(\frac{\sqrt{\lambda_1} + \sqrt{\lambda_2}}{\sqrt{\lambda_1} - \sqrt{\lambda_2}}\right)^2\prod_{x \in (Q_I \cap Q_\phi) \setminus (\infty, \infty)} \left(\frac{\sqrt{\lambda_1 F_2(x)} + \sqrt{\lambda_2 F_1(x)}}{\sqrt{\lambda_1 F_2(x)} - \sqrt{\lambda_2 F_1(x)}}\right)^2 
\end{align}
To finish the computation, it suffices to determine the values of $x$ explicitly.

\subsection{Calculating $x(\phi(\alpha))$ for a 3-isogeny $\phi: E_{\lambda_1} \to E_{\lambda_2}$} We proceed by first computing these $x$-coordinates in a different model for $E_2$ and transform them to points in the Legendre form.

But first, the $3$-torsion points on $E_1$ are given by the following lemma: 

\begin{lem} 
  A point $P_0 = (x_0, y_0)$ is a non-trivial $3$-torsion point on $E_{\lambda}$ if and only if $x_0$ satisfies
  \[3x_0^4 - 4 (\lambda + 1) x_0^3 + 6 \lambda x_0^2 - \lambda^2 = 0.\] 
\end{lem}
\begin{proof}
  Let $\mu$ denote the slope of the tangent line at $P_0$. Then, the $x$- and $y$-co-ordinates of $2P_0$ and $-P_0$ are given by
  \begin{align*}
    x(2P_0) &= (\lambda + 1) + \mu^2 - 2x_0 & x(-P_0) &=  x_0\\
    y(2P_0) &= -y_0 - \mu(x(2P) -  x_0) & y(-P_0)&= -y_0
  \end{align*}
  (as is readily obtained from the chord-tangent description of the addition law on an elliptic curve in Weierstrass form). Now, $P_0$ is a $3$-torsion point if and only if $2P_0 = -P_0$ which is equivalent to
  \begin{align*}
    x(2P_0) = (\lambda + 1) + \mu^2 - 2x_0 &= x_0 \\
     -y_0 - \mu(x(2P_0) -  x_0) &= - y_0 
  \end{align*}
  Note now that the two equations are equivalent. Plugging in $\mu$, viz.,
\[\mu = \frac{3x_0^2 - 2 (\lambda + 1) x_0 + \lambda}{2y_0} \]
 and noting that $y_0^2 = x_0(x_0-1)(x_0 - \lambda)$, the claimed equivalence follows after straightforward algebra. 
\end{proof}

We wish to explicitly compute the Weierstrass model for the target of the isogeny whose kernel is the subgroup $\langle P_0 \rangle$ generated by a non-trivial $3$-torsion point $P_0$ on $E_\lambda$. For this, we use Kohel's algorithm\footnote{The original ideas are due to V\'elu \cite{V71} (as Kohel mentions in his thesis \cite{Koh96}); see also p.~44 of \cite{Elk98}.} \cite{Koh96}.

\begin{lem}
  Let $P_0 = (x_0, y_0)$ denote a non-trivial $3$-torsion point on $E_\lambda$. Then Kohel's Weierstrass model for $E_\lambda/\langle P_0 \rangle$ is given by $y^2 = G_\lambda(x)$ where 
  \begin{align*}
    G_\lambda(x) &= x^3 - (\lambda + 1)x^2 - (30x_0^2 -20(\lambda + 1)x_0 + 9\lambda) x \\
        &\qquad - 70x_0^3 + 80(\lambda + 1)x_0^2 - 2(8\lambda^2 + 37\lambda + 8)x_0 + 8\lambda + 8\lambda^2.
  \end{align*}
  Moreover, if $\phi: E_\lambda \to E_\lambda/\langle P_0 \rangle$ denotes the isogeny whose kernel is $\langle P_0 \rangle$, then
  \begin{align*}
    x(\phi(Q)) &= \frac{4y^2-(6x^2 - 4 (\lambda + 1)x + 2\lambda)(x-x_0) + (3x -2x_0)(x-x_0)^2}{(x-x_0)^2} \\
               &= \frac{(4x^3 - 4 (\lambda + 1) x^2 + 4 \lambda x)-(6x^2 - 4 (\lambda + 1)x + 2\lambda)(x-x_0) + (3x -2x_0)(x-x_0)^2}{(x-x_0)^2} \\
  \end{align*}
  for points $Q = (x(Q), y(Q)) = (x, y)$ on $E_\lambda$. 
\end{lem}
\begin{proof}
  For proof, we refer to Kohel's thesis \cite{Koh96} where this is worked out in greater generality. 
\end{proof}

Now let $\beta: E_{\lambda_1}/ \langle P_0 \rangle \to E_{\lambda_2}$ be the isomorphism between the Weierstrass form $E_{\lambda_1}/ \langle P_0 \rangle$ and its Legendre form. This isomorphism is easy to describe---if $x_0(\lambda_1), x_1(\lambda_1), x_2(\lambda_1)$ are the roots of the polynomial equation $G_{\lambda_1}(x) = 0$, then we may translate $x$ so that $x_0(\lambda_1)$ maps to $0$, then scale $x$ so that $x_1(\lambda_1) - x_0(\lambda_1)$ maps to $1$ (this will involve scaling $y$ correspondingly); one sees that the Legendre parameter of the target of the isogeny $\phi$ is then \[\lambda_2 = \frac{x_3(\lambda_1) - x_0(\lambda_1)}{x_2(\lambda_1) - x_0(\lambda_1)}.\] 
As is clear from the description, the map $\beta$ is harder to write down explicitly over (a suitable extension of) $\mathbf{Q}(\lambda_1)$.

The calculation is complete once we solve the polynomial equation
\[x(\alpha) = x((\beta \circ \phi)(\alpha))\]
for $\alpha \in E_1$.

\bibliographystyle{alpha}
\bibliography{/home/ramesh/Desktop/Bibliography/AlgebraicCycles.bib}

\begin{tabular}[t]{l@{\extracolsep{8em}}l} 
	Kannappan Sampath & Ramesh Sreekantan\\
	Department of Mathematics  & Statistics and Mathematics Department\\
	IISER, Pune  & Indian Statistical Institute \\
	Dr. Homi Bhabha Road & 8th Mile, Mysore Road \\
	Pune 411 008 & Bengaluru 560 059 \\ 
	India & India \\
	kntrichy@gmail.com & rameshsreekantan@gmail.com 
\end{tabular}

\end{document}